\documentclass[11pt]{article}
\usepackage[a4paper]{geometry}
\usepackage{amsthm}
\usepackage{amsmath}
\usepackage{amssymb}
\usepackage{amsfonts}
\usepackage{graphicx}
\usepackage{color}
\usepackage[bf,SL,BF]{subfigure}

\usepackage{epsfig,amsbsy,graphicx,multirow}

\usepackage[all]{xypic}

 \numberwithin{equation}{section}
\def\eref#1{(\ref{#1})}

\def\ve{\varepsilon}
\def\N{\mathbb{N}}
\def\f{\varphi}
\def\ve{\varepsilon}

\def\R{\mathbb{R}}
\def\T{\mathbb{T}}

\def\V{\mathcal{V}}

\def\F{\mathcal{F}}

\def\<{\big\langle}
\def\>{\big\rangle}

\def\diiv{\operatorname{div}}

\def\Tr{\operatorname{Trace}}

\def\det{\operatorname{det}}
\def\Hess{\operatorname{Hess}}

\def\esssup{{\operatorname{esssup}}}

\def\f{{\operatorname{-flux}}}

\newtheorem{Lemma}{Lemma}[section]
\newtheorem{Corollary}{Corollary}[section]
\newtheorem{Theorem}{Theorem}[section]
\newtheorem{Proposition}{Proposition}[section]

\theoremstyle{remark}
\newtheorem{Remark}{Remark}[section]

\theoremstyle{definition}

\newtheorem{Definition}{Definition}[section]

\theoremstyle{definition}

\begin{document}
\title{Flux norm approach   to finite dimensional homogenization approximations with  non-separated scales and high contrast.}

\date{\today}

\author{Leonid Berlyand\footnote{Pennsylvania State University, Department of
Mathematics}  \quad and Houman Owhadi\footnote{Corresponding author. California Institute of
Technology, Applied \& Computational Mathematics, Control \&
Dynamical systems, MC 217-50 Pasadena, CA 91125,
owhadi@caltech.edu}.} \maketitle

\begin{abstract}
We consider linear divergence-form scalar elliptic
equations and  vectorial equations for elasticity with rough
($L^\infty(\Omega)$, $\Omega \subset \R^d$) coefficients $a(x)$ that, in particular,  model media with non-separated scales and high contrast in material properties.   While  the homogenization of PDEs with periodic or ergodic coefficients and well separated scales is now well understood,  we consider here  the  most general case of arbitrary bounded coefficients.  For such problems, we  introduce  explicit and optimal   finite dimensional approximations of solutions that can be viewed as a theoretical Galerkin method with controlled error estimates, analogous to  classical homogenization  approximations.  In particular, this approach allows one to analyze a given medium directly without introducing the mathematical concept of an $\epsilon$ family of media as in  classical homogenization.
We define the flux norm as the $L^2$ norm of the potential part of the fluxes of solutions, which is equivalent to the usual $H^1$-norm.
We show that in the flux  norm, the error associated with approximating, in a properly defined  finite-dimensional  space, the set of solutions of the aforementioned  PDEs  with rough coefficients is  equal to the error associated with approximating the set of solutions of the same type  of PDEs with smooth coefficients in a standard space (e.g., piecewise polynomial).    We refer to this property as the {\it transfer property}.
A simple application of this property is the construction of finite dimensional approximation spaces with errors independent of the regularity and contrast of the coefficients and with optimal and explicit convergence rates.
 This transfer property also provides an alternative to the global harmonic change of coordinates for the homogenization of elliptic operators that can be extended to elasticity equations. The proofs of these homogenization results are based on a new class of elliptic inequalities. These inequalities play the same role in our approach as the div-curl lemma in classical homogenization.
\end{abstract}
\maketitle

\tableofcontents
\paragraph{Acknowledgements.} Part of the research of H. Owhadi is
 supported by the National Nuclear Security Administration
through the Predictive Science Academic Alliance Program. The work of L. Berlyand is supported in part by NSF grant DMS-0708324 and DOE grant DE-FG02-08ER25862.
We would like to thank L. Zhang for the computations associated with figure \ref{figupscaling}.  We also thank  B. Haines, L.  Zhang, and O. Misiats for carefully reading   the manuscript and  providing useful suggestions. We would like to thank Bj\"{o}rn Engquist, Ivo Babu\v{s}ka and John Osborn for useful comments and showing us related and missing references. We are also greatly in debt to Ivo Babu\v{s}ka and John Osborn for carefully reading the manuscript and providing us with very detailed comments and references which have lead to substantial changes. We would also like to thank two anonymous referees for precise and detailed  comments and suggestions.

\section{Introduction}

In this paper we  are interested in finite dimensional approximations of solutions of scalar and vectorial divergence form equations with rough coefficients in $\Omega \subset \R^d$, $d \geq 2$.
More precisely, in the  {\it scalar case}, we  consider the partial differential equation
\begin{equation}\label{scalarproblem0}
\begin{cases}
    -\diiv \Big(a(x)  \nabla u(x)\Big)=f(x) \quad  x \in \Omega; f \in L^2(\Omega), \;   a(x)=\{a_{ij} \in L^{\infty}(\Omega)\}\\
    u=0 \quad \text{on}\quad \partial \Omega,
    \end{cases}
\end{equation}
where $\Omega$ is a bounded subset of $\R^d$ with a smooth boundary (e.g., $C^2$)
and $a$ is symmetric and uniformly elliptic on $\Omega$. It follows
that the eigenvalues of $a$ are uniformly bounded from below and
above by two strictly positive constants, denoted by $\lambda_{\min}(a)$ and
$\lambda_{\max}(a)$.  Precisely, for all $\xi \in \R^d$ and $x\in \Omega$,
\begin{equation}
\lambda_{\min}(a)|\xi|^2 \leq \xi^T a(x) \xi \leq \lambda_{\max}(a)
|\xi|^2.
\end{equation}

In the  {\it vectorial case}, we consider the equilibrium deformation of an inhomogeneous elastic body under a
given load $b \in (L^2(\Omega))^d$, described by
\begin{equation}\label{elasticity}
    \begin{cases}
    -\diiv (C(x): \varepsilon(u))=b(x) \quad &  x \in \Omega\\
    u=0 \quad & \text{on } \partial \Omega,
    \end{cases}
\end{equation}
where $\Omega \subset  \mathbb R^d$ is a bounded domain, $C(x)=\{C_{ijkl}(x)\}$ is a 4th order tensor of elastic
modulus (with the associated symmetries), $u(x) \in  \mathbb R^d
$ is the displacement field, and for $\psi \in (H^1_0(\Omega))^d$, $\varepsilon(\psi)$ is the symmetric part of $\nabla \psi$, namely,
\begin{equation}
\varepsilon_{ij}(\psi)=\frac{1}{2}\Big(\frac{\partial
\psi_i}{\partial x_j}+\frac{\partial \psi_j}{\partial x_i}\Big).
\end{equation}
We assume that $C$ is uniformly
elliptic and $C_{ijkl} \in L^\infty(\Omega)$. It follows that the
eigenvalues of $C$ are uniformly bounded from below and above by two
strictly positive constants, denoted by $\lambda_{\min}(C)$ and
$\lambda_{\max}(C)$.

The analysis of finite dimensional approximations of scalar divergence form elliptic,
parabolic and hyberbolic equations with rough coefficients that  in addition satisfy a Cordes-type condition  in arbitrary dimensions has been performed in
\cite{MR2292954, MR2377253, OwZh06c}. In these works, global harmonic
coordinates are used as a  coordinate transformation. We also refer to
the work of  Babu{\v{s}}ka, Caloz, and Osborn  \cite{BabOsb83, MR1286212} in
which a harmonic change of coordinates is introduced in one-dimensional and
quasi-one-dimensional divergence form elliptic problems.

In essence, this harmonic change of coordinates allows for the mapping of the operator $L_a:=\diiv(a\nabla)$ onto the operator $L_Q:=\diiv(Q\nabla)$ where $Q$ is symmetric positive and divergence-free. This latter property of $Q$ implies that $L_Q$  can be written in both a divergence form and a non-divergence form operator. Using the $W^{2,2}$ regularity of solutions of $L_Q v=f$ (for $f\in L^2$), one is able to obtain homogenization results for the operator $L_a$ in the sense of finite dimensional approximations of its solution space (this relation with homogenization theory will be discussed in detail in section \ref{kjkgjkjgkjgjhg}).

This  harmonic change of coordinates  provides the desired approximation in two-dimensional scalar  problems, but  there is no analog of such  a change of coordinates for vectorial elasticity  equations. One goal of this paper is to obtain an analogous homogenization approximation  without  relying on any coordinate change and therefore   allowing for treatment of both scalar and vectorial problems in a unified framework.

In section \ref{sec1}, we introduce a new norm, called the flux norm, defined as the $L^2$-norm of the potential component of the fluxes of solutions of \eref{scalarproblem0} and \eref{elasticity}. We show that this norm is equivalent to the usual $H^1$-norm. Furthermore, this new norm allows for the transfer of error estimates associated with a given elliptic operator $\diiv(a\nabla)$ and a given approximation space $V$ onto error estimates for another given elliptic operator $\diiv(a'\nabla)$ with another approximation space $V'$ provided that the potential part of the fluxes of elements of $V$ and $V'$ span the same linear space. In this work, this transfer/mapping property will replace the transfer/mapping property associated with a global harmonic change of coordinates.

In section \ref{jkgjkhgjgu}, we show that a  simple and straightforward application of the flux-norm transfer property is to obtain finite dimensional approximation spaces for solutions of \eref{scalarproblem0} and \eref{elasticity} with ``optimal'' approximation errors independent of the regularity and contrast of the coefficients and the regularity of $\partial \Omega$.

Another application of the transfer property of the flux norm is given in  section \ref{lkjlsdhkj} for controlling the approximation error associated with  theoretical discontinuous Galerkin solutions of \eref{scalarproblem0} and \eref{elasticity}. In this context, for elasticity equations, harmonic coordinates are replaced by harmonic displacements.  The estimates introduced in  section \ref{lkjlsdhkj}  are based on mapping onto divergence-free coefficients via the flux-norm and a new class of inequalities introduced in section \ref{inequalities}. We believe that these inequalities are of independent interest for PDE  theory and could be helpful in other problems.

Connections between this work, homogenization theory and other related works will be discussed in section \ref{kjkgjkjgkjgjhg}.

\section{The flux norm and its properties}\label{sec1}
In this section, we will introduce the flux-norm and describe its properties when used as a norm for solutions of \eref{scalarproblem0} (and \eref{elasticity}). This flux-norm is equivalent to the usual $H^1_0(\Omega)$-norm (or $(H^1_0(\Omega))^d$-norm for solutions to the vectorial problem), but leads to error estimates that are independent of the material contrast.
Furthermore, it allows for the transfer of error estimates associated with a given elliptic operator $\diiv(a\nabla)$ and a given approximation space $V$ onto error estimates for another given elliptic operator $\diiv(a'\nabla)$ with another approximation space $V'$ provided that the potential part of the fluxes of elements of $V$ and $V'$ span the same linear space. In \cite{MR2292954}, approximation errors have been obtained for theoretical finite element solutions of \eref{scalarproblem0} with arbitrarily rough coefficients $a$. These approximation errors are based on the mapping of the operator $-\diiv(a\nabla)$ onto an non-divergence form operator $-Q_{i,j}\partial_i \partial_j$ using global harmonic coordinates as a change of coordinates.
It is not clear how to extend this change of coordinates to elasticity equations, whereas the flux-norm approach has a natural extension to systems of equations and can be used to link error estimates on two separate operators.

\subsection{Scalar case.}

\begin{Definition}\label{defWeyl}
 For
 $k\in
(L^2(\Omega))^d$, denote by $k_{pot}$ and $k_{curl}$
 the potential and divergence-free portions of the Weyl-Helmholtz decomposition of $k$. Recall that $k_{pot}$ and $k_{curl}$ are orthogonal with respect to the $L^2$-inner product. $k_{pot}$ is the orthogonal projection of $k$ onto $L^2_{pot}(\Omega)$ defined as the closure of the space
$\{\nabla f\;:\; f\in C_0^\infty(\Omega)\}$ in $(L^2(\Omega))^d$. $k_{curl}$ is the orthogonal projection of $k$ onto $L^2_{curl}(\Omega)$ defined as the closure of the space
$\{\xi\;:\; \xi \in (C^\infty(\Omega))^d\; \diiv(\xi)=0\}$ in $(L^2(\Omega))^d$
\end{Definition}
For $\psi \in H^1_0(\Omega)$, define
\begin{equation}\label{lakdlkjkjd3}
\|\psi\|_{a\f}:=\|(a \nabla \psi)_{pot}\|_{(L^2(\Omega))^d}.
\end{equation}

\paragraph{Motivations for the flux norm}
\begin{itemize}
\item The   $(\cdot)_{\text{pot}}$ in the  $a\f$-norm is explained by the fact that in practice, we are interested in fluxes (of heat, stress, oil, pollutant) entering or exiting a given domain. Furthermore, for a vector field $\xi$, $\int_{\partial\Omega}\xi\cdot n ds=\int_{\Omega}\text{div}(\xi)dx=\int_{\Omega}\text{div}(\xi_{\text{pot}})dx$, which means the flux entering or exiting is determined by the potential part of the vector field.  Thus, as with the energy norm, $\|u\|_a^2:=\int_\Omega (\nabla u)^T a \nabla u$, the flux norm has a natural physical interpretation. An error bound given in the flux-norm shows how well fluxes (of heat or stresses) are approximated.
\item While the energy norm is natural in many problems, we argue that this is no longer the case in the presence of high contrast.
Observe that in \cite{ChuHou09}, contrast independent error estimates are obtained by  renormalizing the energy norm by $\lambda_{\min}(a)$. In  \cite{MR1771781}, the error constants associated with the energy norm are made  independent of the contrast by using terms that are appropriately and explicitly weighted by $a$. These modifications on the energy norm or on the error bounds (expressed in the energy norm) have to be introduced because, in the presence of high contrast in material properties, the energy norm blows up.  Even in the simple case where $a$ is a constant ($a=\alpha I_d$ with $\alpha>0$), the solution of \eref{scalarproblem0} satisfies
  \begin{equation}
  \int_{\Omega}(\nabla u)^T a \nabla u=\frac{1}{\alpha} \big\|\nabla \Delta^{-1} f\big\|_{(L^2(\Omega))^d}^2.
  \end{equation}
    Hence the energy norm squared of the solution of \eref{scalarproblem0} blows up like $1/\alpha$ as $\alpha\downarrow 0$ whereas its flux-norm is independent of $\alpha$ (because $(a\nabla u)_{pot}= \nabla \Delta^{-1} f$)
    \begin{equation}\label{kjhkdhjdhjdje}
    \begin{split}
  \|u\|_{a\f}= \big\|\nabla \Delta^{-1} f\big\|_{(L^2(\Omega))^d}.
  \end{split}
  \end{equation}
 Equation \eref{kjhkdhjdhjdje} remains valid even when $a$ is not a constant (this is a consequence of the transfer property, see Corollary \ref{sjhgdkjshgd}).
   In reservoir modeling, fluxes of oil and water are the main quantities of interest to be approximated correctly. The energy norm is less relevant due to high contrast and has been modified (in \cite{ChuHou09} for instance) in order to avoid possible blow up.

\item Similar considerations of convergence in terms energies and fluxes are present in classical homogenization theory. Indeed, the convergence of solutions of $-\diiv(a^\epsilon \nabla u^\epsilon)=f$ can be expressed in terms of convergence of energies in the context of  $\Gamma$-convergence \cite{MR630747, MR1968440} (and its variational formulation) or in the terms (of weak) convergence of fluxes in $G$ or $H$-convergence \cite{Mur78,Gio75, MR0477444, MR0240443, MR506997} ($a^\epsilon \nabla u^\epsilon \rightarrow a^0 \nabla u^0$). Here, weak $L^2$ convergence of fluxes is used and no flux norm is necessary unlike in our study, where it  arises naturally.
\end{itemize}

\begin{Proposition}\label{Prop0}
$\|.\|_{a\f}$ is a norm on $H^1_0(\Omega)$. Furthermore, for all $\psi
\in H^1_0(\Omega)$
\begin{equation}\label{skkshg}
\lambda_{\min}(a) \|\nabla \psi\|_{(L^2(\Omega))^d} \leq
\|\psi\|_{a\f}\leq \lambda_{\max}(a) \|\nabla \psi\|_{(L^2(\Omega))^d}
\end{equation}
\end{Proposition}
\begin{proof}
The proof of the left hand side of inequality \eref{skkshg} follows
by observing that
\begin{equation}
\int_\Omega (\nabla \psi)^T a \nabla \psi =\int_\Omega (\nabla \psi)^T (a \nabla \psi)_{pot}
\end{equation}
from which we deduce by Cauchy-Schwarz  inequality that
\begin{equation}
\int_\Omega (\nabla \psi)^T a \nabla \psi \leq \|\nabla
\psi\|_{L^2(\Omega)} \|\psi\|_{a\f}.
\end{equation}
\end{proof}

The proof of the main theorem  of this section will require
\begin{Lemma}\label{Prop1}
Let $V$ be a finite dimensional linear subspace of $H^1_0(\Omega)$. For $f\in L^2(\Omega)$, let $u$ be the solution of
\eref{scalarproblem0}. Then,
\begin{equation}\label{sids5sawqeyassdud}
\sup_{f \in L^2(\Omega)} \inf_{v\in V} \frac{ \|u-v\|_{a\f}}{
\|f\|_{L^2(\Omega)}}= \sup_{w \in H^2(\Omega)\cap H^1_0(\Omega)}
\inf_{v \in V} \frac{\|(\nabla w-a \nabla
v)_{pot}\|_{(L^2(\Omega))^d}}{\|\Delta w\|_{L^2(\Omega)}}
\end{equation}
\end{Lemma}

\begin{proof}
Since $f\in L^2(\Omega)$, it is known that there exists $w \in
H^2(\Omega)\cap H^1_0(\Omega)$ such that
\begin{equation}
\begin{cases}
    -\Delta w=f \quad  x \in \Omega \\
    w=0 \quad \text{on}\quad \partial \Omega.
    \end{cases}
\end{equation}
We conclude by observing that for $v\in V$,
\begin{equation}\label{skjkjskhe}
 \|(\nabla w-a \nabla
v)_{pot}\|_{(L^2(\Omega))^d}= \|(a \nabla u-a \nabla
v)_{pot}\|_{(L^2(\Omega))^d}.
\end{equation}
\end{proof}

For $V$, a finite dimensional linear subspace of $H^1_0(\Omega)$, we define
\begin{equation}\label{ksjjseddesel3}
(\diiv a \nabla V):=\operatorname{span} \{\diiv(a\nabla v)\,:\,  v\in V\}.
\end{equation}
Note that $(\diiv a \nabla V)$ is a finite dimensional subspace of $H^{-1}(\Omega)$.
\smallskip

The following theorem establishes   the transfer property of the flux norm which is pivotal for our analysis.

\begin{Theorem} \label{sdjhskjdhskdhkhje}{\bf (Transfer property of the flux norm)}
 Let $V'$ and $V$ be finite-dimensional subspaces of $H^1_0(\Omega)$. For $f\in L^2(\Omega)$ let $u$ be the solution of
\eref{scalarproblem0} with conductivity $a$ and  $u'$ be the solution of
\eref{scalarproblem0} with conductivity $a'$. If $(\diiv a \nabla V)=(\diiv a' \nabla V')$, then
\begin{equation}\label{sidasasedsssddaud}
\sup_{f \in L^2(\Omega)} \inf_{v\in V}  \frac{ \|u - v\|_{a\f}}{
\|f\|_{L^2(\Omega)}}= \sup_{f \in L^2(\Omega)} \inf_{v\in V'}  \frac{ \|u' - v\|_{a'\f}}{
\|f\|_{L^2(\Omega)}}.
\end{equation}

\end{Theorem}
\begin{Remark} \label{transferproperty}The usefulness of \eqref{sidasasedsssddaud} can be illustrated by considering $a'=I$ so that
$\diiv a' \nabla = \Delta$. Then $u' \in H^2$ and therefore  $V'$  can be chosen as, e.g.,  the standard piecewise linear  FEM space with nodal basis $\{\varphi_i\}$. The space $V$ is then defined by its basis $\{\psi_i\}$ determined by
\begin{equation}
\diiv (a \nabla \psi_i) = \Delta \varphi_i
\end{equation}
with  Dirichlet boundary conditions (see details in section \ref{khdkhdlkwhe}). Furthermore,  equation \eqref{sidasasedsssddaud}  shows that  the error estimate  for a problem with arbitrarily rough coefficients is  equal to  the well-known   error estimate for  the Laplace equation.
\end{Remark}

\begin{Remark}
Equation \eref{sidasasedsssddaud} remains valid without the supremum in $f$. More precisely writing $u$ and $u'$ the solutions of
\eref{scalarproblem0} with conductivities $a$ and $a'$ and the same right hand side $f\in L^2(\Omega)$, one has
\begin{equation}\label{slsdhsasedsssBddaud}
\inf_{v\in V}   \|u - v\|_{a\f}=  \inf_{v\in V'}  \|u' - v\|_{a'\f}.
\end{equation}
Equation \eref{slsdhsasedsssBddaud} is obtained by observing that
\begin{equation}\label{slsdhsghdabbud}
 \|u - v\|_{a\f}=  \big\|\nabla \Delta^{-1}(f+\diiv(a\nabla v))\big\|_{L^2(\Omega)}
\end{equation}
\end{Remark}

\begin{Corollary}\label{sjhgdkjshgd} Let $X$ and $V$ be finite-dimensional subspaces of $H^1_0(\Omega)$. For $f\in L^2(\Omega)$ let $u$ be the solution of
\eref{scalarproblem0} with conductivity $a$. If $(\diiv a \nabla V)=(\diiv  \nabla X)$ then
\begin{equation}\label{sidasasedewdsssddaud}
\sup_{f \in L^2(\Omega)} \inf_{v\in V}  \frac{ \|u - v\|_{a\f}}{
\|f\|_{L^2(\Omega)}}= \sup_{w\in H^1_0(\Omega)\cap H^2(\Omega)} \inf_{v\in X}  \frac{ \|\nabla w - \nabla v\|_{(L^2(\Omega))^d}}{
\|\Delta w\|_{L^2(\Omega)}}
\end{equation}
\end{Corollary}

Equation \eqref{sidasasedewdsssddaud} can be obtained by  setting $a'=I$ in theorem \ref{sdjhskjdhskdhkhje} and applying lemma \ref{Prop1}.

Theorem \ref{sdjhskjdhskdhkhje} is  obtained from the following proposition by noting that the
right hand side of equation \eqref{sidasawqeyadsedsssddaud} is the same for pairs $(a,V)$ and $(a',V')$
whenever $\diiv (a \nabla V) = \diiv (a' \nabla V')$.
\begin{Proposition}\label{Csorsedfesol4}
For $f\in L^2(\Omega)$ let $u$ be the solution of
\eref{scalarproblem0}. Then,
\begin{equation}\label{sidasawqeyadsedsssddaud}
\sup_{f \in L^2(\Omega)} \inf_{v\in V}  \frac{ \|u - v\|_{a\f}}{
\|f\|_{L^2(\Omega)}}= \sup_{z \in (\diiv a \nabla V)^{\perp}}\frac{\|
z\|_{L^2(\Omega)}}{\|\nabla z\|_{(L^2(\Omega))^d}},
\end{equation}
where
\begin{equation}\label{ksjjsedderesesel3}
(\diiv a \nabla V)^{\perp}:=\{z\in H^1_0(\Omega)\,:\, \forall v\in V, (\nabla z
,a\nabla v)=0 \}.
\end{equation}
\end{Proposition}

\begin{proof}
For $w \in H^2(\Omega)$, define
\begin{equation}\label{sidseedsaswceweawyassdud}
J(w):=\inf_{v\in V}\|(\nabla w-a \nabla
v)_{pot}\|_{(L^2(\Omega))^d}.
\end{equation}
Observe that
\begin{equation}\label{sidseewe3daeswsyedsasesdud}
J(w)=\inf_{v\in V, \xi\in (L^2(\R^d))^d\,:\, \diiv(\xi)=0}\|\nabla w-a
\nabla v-\xi\|_{(L^2(\Omega))^d}.
\end{equation}
Additionally, observing that the space spanned by $\nabla z$ for $z\in(\diiv a \nabla V)^{\perp}$ is
the orthogonal complement (in $(L^2(\Omega))^d$) of the space
spanned by $a \nabla v+\xi$, we obtain that
\begin{equation}\label{sidseewe3daswsweyasesdud}
J(w)=\sup_{z \in (\diiv a \nabla V)^{\perp}}\frac{(\nabla w,\nabla z)}{\|\nabla
z\|_{(L^2(\Omega))^d}}.
\end{equation}
Integrating by parts and applying the Cauchy-Schwarz  inequality
yields
\begin{equation}\label{sidseewe3daswsyaewwlklesesdud}
J(w)\leq \|\Delta w\|_{L^2(\Omega)} \sup_{z \in (\diiv a \nabla V)^{\perp}}\frac{\|
z\|_{L^2(\Omega)}}{\|\nabla z\|_{(L^2(\Omega))^d}}.
\end{equation}
which proves
\begin{equation}\label{sidasawqesdsyadsedsssddaud}
\sup_{f \in L^2(\Omega)} \inf_{v\in V}  \frac{ \|u - v\|_{a\f}}{
\|f\|_{L^2(\Omega)}}\leq \sup_{z \in (\diiv a \nabla V)^{\perp}}\frac{\|
z\|_{L^2(\Omega)}}{\|\nabla z\|_{(L^2(\Omega))^d}},
\end{equation}
Dividing by $\| \Delta w \|_{L^2(\Omega)}$, integrating by parts, and taking the supremum
over $w \in H^2(\Omega)\cap H^1_0(\Omega)$, we get
\begin{equation}\label{sidseewe3dddjkjkassdsdwsyasesdud}
\sup_{w \in H^2(\Omega)\cap H^1_0(\Omega)} \frac{J(w)}{\| \Delta w \|_{L^2(\Omega)}}=\sup_{z \in (\diiv a \nabla V)^{\perp}} \sup_{w \in H^2(\Omega)\cap H^1_0(\Omega)} -\frac{(\Delta w, z)}{\|\nabla
z\|_{(L^2(\Omega))^d} \| \Delta w \|_{L^2(\Omega)}}.
\end{equation}
we conclude the theorem by choosing $-\Delta w=z$.
\end{proof}

The transfer property \eqref{sidasasedsssddaud}  for  solutions can be complemented by an analogous property for fluxes. To this end, for  a finite dimensional linear subspace
$\V\subset (L^2(\Omega))^d$ define
\begin{equation}\label{ksjjsedddddeesel3}
(\diiv a \V):= \{\diiv(a\zeta)\,:\,  \zeta\in \V\}.
\end{equation}
Observe that $(\diiv a \V)$ is a finite dimensional subspace of $H^{-1}(\Omega)$. The proof of the following theorem is  similar to the proof of theorem \ref{sdjhskjdhskdhkhje}.

\begin{Theorem}\label{sdjhskjdhskdhsdsdsdsxkhje} {\bf (Transfer property for fluxes)}
Let $\V'$ and $\V$ be finite-dimensional subspaces of $(L^2(\Omega))^d$. For $f\in L^2(\Omega)$ let $u$ be the solution of
\eref{scalarproblem0} with conductivity $a$ and  $u'$ be the solution of
\eref{scalarproblem0} with conductivity $a'$. If $(\diiv a \V)=(\diiv a' \V')$ then
\begin{equation}\label{sidasasedsssdeeewdaud}
\sup_{f \in L^2(\Omega)} \inf_{\zeta\in \V}  \frac{ \|(a(\nabla u - \zeta))_{pot}\|_{(L^2(\Omega))^d}}{
\|f\|_{L^2(\Omega)}}= \sup_{f \in L^2(\Omega)} \inf_{\zeta\in \V'}  \frac{ \|(a'(\nabla u' - \zeta))_{pot}\|_{(L^2(\Omega))^d}}{
\|f\|_{L^2(\Omega)}}
\end{equation}
\end{Theorem}
Theorem \ref{sdjhskjdhskdhsdsdsdsxkhje} will be used in section \ref{lkjlsdhkj} for obtaining error estimates on
 theoretical non-conforming Galerkin solutions of \eref{scalarproblem0}.

 \begin{Corollary} Let $\V$ be a finite-dimensional subspace of $(L^2(\Omega))^d$ and
$X$ a finite-dimensional subspace of $H^1_0(\Omega)$. For $f\in L^2(\Omega)$ let $u$ be the solution of
\eref{scalarproblem0} with conductivity $a$. If $(\diiv a \V)=(\diiv  \nabla X)$ then
\begin{equation}\label{sidasasedsssffxddaud}
\sup_{f \in L^2(\Omega)} \inf_{\zeta\in \V}  \frac{ \|(a(\nabla u - \zeta))_{pot}\|_{(L^2(\Omega))^d}}{
\|f\|_{L^2(\Omega)}}= \sup_{w\in H^1_0(\Omega)\cap H^2(\Omega)} \inf_{v\in X}  \frac{ \|\nabla w - \nabla v\|_{(L^2(\Omega))^d}}{
\|\Delta w\|_{L^2(\Omega)}}
\end{equation}
\end{Corollary}

\begin{Remark}
The analysis performed in this section and in the following one can be naturally extended to other types of boundary conditions (nonzero Neumann or Dirichlet). To support our claim, we will provide this extension in the scalar case with non-zero Neumann boundary conditions. We refer to subsection
\ref{NonZeroBC} for that extension.
\end{Remark}

\subsection{Vectorial case.}
 For
$k\in
(L^2(\Omega))^{d\times d}$, denote by $k_{pot}$
 the potential portion of the Weyl-Helmholtz decomposition of $k$
(the orthogonal projection of $k$ onto the closure of the space
$\{\nabla f\;:\; f\in (C_0^\infty(\Omega))^d\}$ in
$(L^2(\Omega))^{d\times d}$).
Define
\begin{equation}\label{jdhkwjhdje}
\|\psi\|_{C\f}:=\|(C:
\varepsilon(\psi))_{pot}\|_{(L^2(\Omega))^{d\times d}}.
\end{equation}

\begin{Remark}
Because of the symmetries of the elasticity tensor $C$, one has $\forall f\in (C_0^\infty(\Omega))^d$
\begin{equation}
\Big(\nabla f,(C:
\varepsilon(\psi))_{pot}\Big)_{(L^2(\Omega))^{d\times d}}=\Big(\varepsilon( f),(C:
\varepsilon(\psi))_{pot}\Big)_{(L^2(\Omega))^{d\times d}}
\end{equation}
from which it follows that definition \ref{jdhkwjhdje} would be the same if the projection was made on the space of symmetrized gradients.
\end{Remark}

\begin{Proposition}\label{Prop2a}
$\|.\|_{C\f}$ is a norm on $(H^1_0(\Omega))^{d}$.
Furthermore, for all
$\psi \in (H^1_0(\Omega))^d$
\begin{equation}\label{skkzshg}
\lambda_{\min}(C) \|\varepsilon(\psi)\|_{(L^2(\Omega))^{d\times d}}
\leq \|\psi\|_{C\f} \leq \lambda_{\max}(C)
\|\varepsilon(\psi)\|_{(L^2(\Omega))^{d\times d}}.
\end{equation}
\end{Proposition}
\begin{proof}
The proof of the left hand side of inequality \eref{skkzshg} follows
by observing that
\begin{equation}
\int_\Omega ( \varepsilon(\psi))^T :C: \varepsilon( \psi) \leq
\|\varepsilon(\psi)\|_{(L^2(\Omega))^{d\times d}} \|\psi\|_{C\f}.
\end{equation}
The fact that $\|\psi\|_{C\f}$ is a norm follows from the left hand
side of inequality \eref{skkzshg} and Korn's inequality
\cite{MR961258}: i.e., for all $\psi \in (H^1_0(\Omega))^d$,
\begin{equation}\label{sksekdaseshg}
 \|\nabla \psi\|_{(L^2(\Omega))^{d\times d}} \leq \sqrt{2} \|\varepsilon(\psi)\|_{(L^2(\Omega))^{d\times
d}}.
\end{equation}
\end{proof}

For $V$, a finite dimensional linear subspace of $(H^1_0(\Omega))^d$, we define
\begin{equation}\label{ksjjseddesdsdsel3}
(\diiv C :\varepsilon(V)):=\operatorname{span} \{\diiv(C:\varepsilon(v))\,:\,  v\in V\}.
\end{equation}
Observe that $(\diiv C:\varepsilon(V))$ is a finite dimensional subspace of $(H^{-1}(\Omega))^d$. Similarly for
$X$, a finite dimensional linear subspace of $(H^1_0(\Omega))^d$, we define
\begin{equation}\label{ksjjseddesddddesdsel3}
\Delta X:=\operatorname{span} \{\Delta v\,:\,  v\in X\}.
\end{equation}

\begin{Theorem}\label{sdjhskzzhdsdkhje}
Let $V'$ and $V$ be finite-dimensional subspaces of $(H^1_0(\Omega))^d$. For $b\in (L^2(\Omega))^d$ let $u$ be the solution of
\eref{elasticity} with elasticity $C$ and  $u'$ be the solution of
\eref{elasticity} with elasticity $C'$. If $(\diiv C: \varepsilon( V))=(\diiv C': \varepsilon( V'))$ then
\begin{equation}\label{sizsssddzaud}
\sup_{b \in (L^2(\Omega))^d} \inf_{v\in V}  \frac{ \|u - v\|_{C\f}}{
\|b\|_{(L^2(\Omega))^d}}= \sup_{b \in (L^2(\Omega))^d} \inf_{v\in V'}  \frac{ \|u' - v\|_{C'\f}}{
\|b\|_{(L^2(\Omega))^d}}
\end{equation}
\end{Theorem}

\begin{Corollary}\label{corollary1}
 Let $X$ and $V$ be finite-dimensional subspaces of $(H^1_0(\Omega))^d$. For $b\in (L^2(\Omega))^d$ let $u$ be the solution of
\eref{elasticity} with elasticity tensor $C$. If $(\diiv C: \varepsilon( V))=\Delta X$ then
\begin{equation}\label{sidasasedzsddaud}
\sup_{b \in (L^2(\Omega))^d} \inf_{v\in V}  \frac{ \|u - v\|_{C\f}}{
\|b\|_{(L^2(\Omega))^d}}= \sup_{w\in (H^1_0(\Omega)\cap H^2(\Omega))^d} \inf_{v\in X}  \frac{ \|\nabla w - \nabla v\|_{(L^2(\Omega))^{d\times d}}}{
\|\Delta w\|_{(L^2(\Omega))^d}}
\end{equation}
\end{Corollary}

The proof of theorem \ref{sdjhskzzhdsdkhje} is analogous to the proof of theorem \ref{sdjhskjdhskdhkhje}.

For $\V$  a finite dimensional linear subspace of
$(L^2(\Omega))^{d\times d}$ we define
\begin{equation}\label{ksjjsedddddddeeesel3}
(\diiv C:\V):=\operatorname{span} \{\diiv(C:\zeta)\,:\,  \zeta\in \V\}.
\end{equation}
Observe that $(\diiv C: \V)$ is a finite dimensional subspace of $(H^{-1}(\Omega))^d$. The proof of the following theorem is analogous to the proof of theorem \ref{sdjhskjdhskdhkhje}.

\begin{Theorem}\label{sdjhskjddshskdhsdsdsdsxkhje}
Let $\V'$ and $\V$ be finite-dimensional subspaces of $(L^2(\Omega))^{d\times d}$. For $b\in (L^2(\Omega))^{d}$ let $u$ be the solution of
\eref{elasticity} with conductivity $C$ and  $u'$ be the solution of
\eref{elasticity} with conductivity $C'$. If $(\diiv C: \V)=(\diiv C': \V')$ then
\begin{equation}\label{sidasasedsdssdeeewdaud}
\sup_{b \in (L^2(\Omega))^{d}} \inf_{\zeta\in \V}  \frac{ \|(C:(\varepsilon( u) - \zeta))_{pot}\|_{(L^2(\Omega))^{d\times d}}}{
\|b\|_{(L^2(\Omega))^d}}= \sup_{b \in (L^2(\Omega))^d} \inf_{\zeta\in \V'}  \frac{ \|(C':(\varepsilon( u) - \zeta))_{pot}\|_{(L^2(\Omega))^{d\times d}}}{
\|b\|_{(L^2(\Omega))^d}}
\end{equation}

\end{Theorem}

\begin{Corollary}
Let $\V$ be a finite-dimensional subspace of $(L^2(\Omega))^{d\times d}$ and
$X$ a finite-dimensional subspace of $(H^1_0(\Omega))^d$. For $b\in (L^2(\Omega))^d$ let $u$ be the solution of
\eref{elasticity} with elasticity $C$. If $(\diiv C: \V)=(\Delta X)$ then
\begin{equation}\label{sidasasedsfdssffxddaud}
\sup_{b \in L^2(\Omega)} \inf_{\zeta\in \V}  \frac{ \|(C:(\varepsilon (u) - \zeta))_{pot}\|_{(L^2(\Omega))^{d\times d}}}{
\|b\|_{(L^2(\Omega))^d}}= \sup_{w\in H^1_0(\Omega)\cap H^2(\Omega)} \inf_{v\in X}  \frac{ \|\nabla w - \nabla v\|_{(L^2(\Omega))^{d\times d}}}{
\|\Delta w\|_{(L^2(\Omega))^d}}
\end{equation}

\end{Corollary}

\section{Application to theoretical finite element methods with accuracy independent of material contrast.}\label{jkgjkhgjgu}
In this section, we will show how, as a very simple and straightforward application, the flux norm can be used to construct finite dimensional approximation spaces for solutions of \eref{scalarproblem0} and \eref{elasticity} with errors independent of the regularity and contrast of the coefficients and the regularity of $\partial \Omega$ (for the basis defined in subsection \ref{khdkhdlsdde3wkwhe}).
A similar  approximation problem can be found in the work of Melenk \cite{MR1766938}, where subsets of $L^2$ such as  piecewise discontinuous polynomials have been used  as an approximation basis (for the right hand side of \eref{scalarproblem0}). The main difference between \cite{MR1766938} and this section lies in the introduction  of  the flux-norm ($\|.\|_{a\f}$), which plays  a key role in our analysis, since the approximation error (in $\|.\|_{a\f}$-norm) of the space $V_h$  on solutions of the operator $\diiv(a\nabla)$  is equal to the approximation error (in $\|.\|_{a'\f}$-norm) of the space $V_h'$  on solutions of the operator $\diiv(a'\nabla)$ provided that $\diiv(a\nabla V_h)=\diiv(a'\nabla V_h')$. Moreover, this  allows us to obtain an {\it explicit and optimal} constant in the  rate of convergence (theorem \ref{Corol1} and \ref{Corol1del}). To our knowledge, no explicit optimal error constant has been obtained for finite-dimensional approximations of the solution space of \eref{scalarproblem0}.  This question of optimal approximation with respect to a linear finite dimensional space is related to  the Kolmogorov n-width \cite{Pinkus85}, which measures how accurately a given set of functions can be approximated by linear spaces of dimension $n$ in a given norm. A surprising result of the theory of n-widths is the non-uniqueness of the space realizing the optimal approximation \cite{Pinkus85}.
A related work is also \cite{MR2430615}, in which errors in approximations to  solutions of $\diiv(a\nabla u)=0$ from  linear spaces generated by a finite set of boundary conditions are analyzed as functions of the distance to the boundary (the penetration function).

\subsection{Scalar divergence form   equation}\label{kjgfgfgh}

\subsubsection{Approximation with piecewise linear nodal basis functions of a regular tessellation of $\Omega$}\label{khdkhdlkwhe}

Let $\Omega_h$ be a regular tessellation of $\Omega$ of resolution $h$ (we refer to \cite{BrSc02}). Let $\mathcal{L}_0^h$ be the set of piecewise linear functions on $\Omega_h$ with Dirichlet boundary conditions.  Denote by $\varphi_k$ the piecewise linear nodal basis elements of $\mathcal{L}_0^h$, which are {\it localized} (the support of $\varphi_k$ is the union of simplices contiguous to the node $k$). Here, we will express the error estimate  in terms of $h$  to emphasize the analogy with classical FEM (it could be expressed in terms of $N(h)$, see below if needed).

Let $\Phi_k$ be the functions
associated with the  piecewise linear nodal basis elements  $\varphi_k$
through the equation
\begin{equation}\label{harmonsedcalwsarm}
    \left\{\begin{array}{ll}
   -\diiv\left( a(x) \nabla \Phi_k (x)\right) = \Delta \varphi_k\; & \text{in } \Omega\\
    \Phi_k=0  & \mbox{on}\;  \partial \Omega\end{array}.\right.
\end{equation}

Define
\begin{equation}\label{harmseeddcalwddesar}
V_h:=\operatorname{span}\{\Phi_k\},
\end{equation}

\begin{Theorem}\label{jhdjkhsgdkjgeyhde}
 For any $f\in L^2(\Omega)$, let $u$ be the
solution of \eref{scalarproblem0}. Then,
\begin{equation}\label{sids5sassddedawqeyassdud}
\sup_{f \in L^2(\Omega)} \inf_{v\in
V_h} \frac{ \|u-v\|_{a\f}}{  \|f\|_{L^2(\Omega)}}\leq C h
\end{equation}
where $C$ depends only on $\Omega$ and the aspect ratios of the simplices of $\Omega_h$.
\end{Theorem}

\begin{proof}
Theorem \ref{jhdjkhsgdkjgeyhde} is a straightforward application of the equation \eref{sidasasedewdsssddaud} and the fact that one can approximate
$H^2$ functions by functions from $\mathcal{L}_0^h$ in the $H^1$ norm with $\mathcal{O}(h)$ accuracy (since $\partial \Omega$ is of class $C^2$ solutions of the Laplace-Dirichlet operator with $L^2$ right hand sides are in $H^2$, we refer to \cite{BrSc02}).
\end{proof}

\begin{Corollary}\label{jhdjkhsgsddssdgeyhde}
 For $f\in L^2(\Omega)$, let $u$ be the
solution of \eref{scalarproblem0} in $H^1_0(\Omega)$ and $u_h$ the finite element solution of \eref{scalarproblem0} in $V_h$. Then,
\begin{equation}\label{sids5scdqeyassdud}
\sup_{f \in L^2(\Omega)} \frac{ \|u-u_h\|_{H^1_0(\Omega)}}{  \|f\|_{L^2(\Omega)}}\leq \frac{C}{\lambda_{\min}(a)} h
\end{equation}
where $C$ depends only on $\Omega$ and the aspect ratios of the simplices of $\Omega_h$.
\end{Corollary}

\begin{proof}
Corollary \ref{jhdjkhsgsddssdgeyhde} is a straightforward application of theorem \ref{jhdjkhsgdkjgeyhde} and inequality \eref{skkshg}.
\end{proof}

Let $Q$ be a  symmetric, uniformly elliptic, {\it divergence-free} (as defined in section \ref{inequalities})  matrix with entries in $L^\infty(\Omega)$.  We note that this matrix will be chosen below so that the solutions of $\diiv Q \nabla u=f$ are in $H^2(\Omega)$ if $f \in L^2(\Omega)$ and therefore can be approximated by functions from $\mathcal{L}_0^h$ in $H^1$ norm with $\mathcal{O}(h)$ accuracy. It follows from \cite{MR1648351} that this is not possible for  the solutions of \eqref{scalarproblem0}.  In particular, in some cases $Q$ can be chosen to be the identity.

Let $\Phi_k^Q$ be the functions
associated with the  piecewise linear nodal basis elements  $\varphi_k$
through the equation
\begin{equation}\label{harmonsefdfddcalwsarm}
    \left\{\begin{array}{ll}
   -\diiv\left( a(x) \nabla \Phi_k^Q(x)\right) = \diiv(Q\nabla \varphi_k)\; & \text{in } \Omega\\
    \Phi_k^Q=0  & \mbox{on}\;  \partial \Omega\end{array}.\right.
\end{equation}

Define
\begin{equation}\label{harmseeddcalwddesabr}
V_h^Q:=\operatorname{span}\{\Phi_k^Q\},
\end{equation}

\begin{Theorem}
 For $f\in L^2(\Omega)$, let $u$ be the
solution of \eref{scalarproblem0} in $H^1_0(\Omega)$ and $u_h$ the finite element solution of \eref{scalarproblem0} in $V_h^Q$. If $Q$ satisfies one of the inequalities of theorem \ref{kshkheie} or theorem \ref{Prop6s} then
\begin{equation}\label{sifrds5scdqeyassdsdeefrfud}
\sup_{f \in L^2(\Omega)} \frac{ \|u-u_h\|_{H^1_0(\Omega)}}{  \|f\|_{L^2(\Omega)}}\leq \frac{C}{\lambda_{\min}(a)} h
\end{equation}
where $C$ depends only on $\Omega$ and the aspect ratios of the simplices of $\Omega_h$.
\end{Theorem}
\begin{proof}
The proof follows from the fact that if $Q$ satisfies one of the inequalities of theorem \ref{kshkheie} or theorem \ref{Prop6s} then solutions of
$-\diiv(Q\nabla u)=f$ with Dirichlet boundary conditions are in $H^2$. The rest of the proof is similar to that of the previous corollary.
\end{proof}

\subsubsection{Approximation with eigenfunctions of the Laplace-Dirichlet operator.}\label{khdkhdlsdde3wkwhe}
In this sub-section, we assume the minimal regularity condition ($C^2$) on the boundary $\partial \Omega$ such that the Weyl  formula holds (we refer to \cite{NetSaf05} and references therein).

Denote by $\Psi_k$ the eigenfunctions associated with the
Laplace-Dirichlet operator in $\Omega$ and $\lambda_k$ the
associated eigenvalues--i.e., for $k\in \N^*=\{1,2, \cdots\}$
\begin{equation}\label{kjskske23}
\begin{cases}
    -\Delta \Psi_k=\lambda_k \Psi_k \quad  x \in \Omega \\
    \Psi_k=0 \quad \text{on}\quad \partial \Omega.
    \end{cases}
\end{equation}
We assume that the eigenvalues are ordered--i.e.,
$\lambda_k \leq
\lambda_{k+1}$.

Let $\theta_k$ be the functions
associated with the Laplace-Dirichlet eigenfunctions $\Psi_k$
\eref{kjskske23} through the equation
\begin{equation}\label{harmonsedcalwsar}
    \left\{\begin{array}{ll}
   -\diiv\left( a(x) \nabla \theta_k (x)\right) = \lambda_k \Psi_k\; & \text{in } \Omega\\
    \theta_k=0  & \mbox{on}\;  \partial \Omega\end{array}.\right.
\end{equation}

Here, $\lambda_k$ is introduced on the right hand side of
\eref{harmonsedcalwsar} in order to normalize $\theta_k$ ($\theta_k=\Psi_k$, if $a(x)=I$) and can be otherwise
ignored since only  the span of $\{\theta_k\}$ matters.
 Define
\begin{equation}\label{harmseeddcalwsar}
\Theta_h:=\operatorname{span}\{\theta_1,\ldots,\theta_{N(h)}\},
\end{equation}
where $N(h)$ is the integer part of $|\Omega|/h^d$. The
motivation behind our definition of $\Theta_h$ is that its dimension
corresponds to the number of degrees of freedom of piecewise linear
functions on a regular triangulation (tessellation) of $\Omega$ of
resolution $h$.

\begin{Theorem}\label{Corol1}
 For $f\in L^2(\Omega)$, let $u$ be the
solution of \eref{scalarproblem0}. Then,
\begin{itemize}
\item
\begin{equation}\label{sids5sasawqeyassdud}
\lim_{h\rightarrow 0}\sup_{f \in L^2(\Omega)} \inf_{v\in
\Theta_h} \frac{ \|u-v\|_{a\f}}{ h \|f\|_{L^2(\Omega)}}=\frac{1}{2\sqrt{\pi}}
\Big(\frac{1}{\Gamma(1+\frac{d}{2})}\Big)^\frac{1}{d}.
\end{equation}
Furthermore, the space $\Theta_h$ leads (asymptotically as $h \to 0$) to the smallest possible constant in the right hand side of \eref{sids5sasawqeyassdud} among all subspaces of $H^1_0(\Omega)$ with $N(h)$, the integer part of $|\Omega|/h^d$, elements.
\item
\begin{equation}\label{sidasaeddeedsssddauddd999}
\inf_{V,\text{dim}(V)=N}\sup_{f \in L^2(\Omega)} \inf_{v\in V}  \frac{ \|u - v\|_{a\f}}{
\|f\|_{L^2(\Omega)}}= \frac{1}{2 \sqrt{\pi}}
\Big(\frac{|\Omega|}{\Gamma(1+\frac{d}{2})N}\Big)^\frac{1}{d} \big(1+\epsilon(N)\big)
\end{equation}
where the infimum is taken with respect to all subspaces of $H^1_0(\Omega)$ with $N$ elements and $\epsilon(N)$ is converges to zero as $N\rightarrow \infty$.
\end{itemize}
\end{Theorem}

\begin{Remark}  The constants in the right hand side of \eqref{sids5sasawqeyassdud} and \eref{sidasaeddeedsssddauddd999} are the  classical Kolmogorov   $n$-width $d_n(A, X)$, understood in the ``asymptotic'' sense (as $h\to 0$ for \eref{sids5sasawqeyassdud} and $N\rightarrow \infty$ for \eref{sidasaeddeedsssddauddd999}) because the Weyl formula is asymtotic. Recall that the $n$-width measures how accurately  a given set of functions $A \subset X$ can be approximated by linear spaces $E_n$ of dimension $n$.  Writing $d_n(A,X)$ the $n$-width measure, it is defined by
$$d_n(A,X):=\inf_{E_n}\sup_{w\in A}\inf_{g\in E_n}\|w-g\|_X$$ for a  normed linear space $X$. In our case $X=H^1_0(\Omega)$, $A$ being the set of all solutions of \eqref{scalarproblem0} as $f$ spans $L^2(\Omega)$ for a given $a(x)$  and $\Omega$. It should be observed there is a  slight difference with classical Kolmogorov  $n$-width, indeed the flux norm
 $\|.\|_{a\f}$ used in \eref{sids5sasawqeyassdud} depends on $a$ (as opposed to the $H^1_0(\Omega)$-norm ).
A surprising result of the theory of n-widths \cite{Pinkus85} is that  the space realizing the optimal approximation  is not unique, therefore there may be subspaces, other than $\Theta_h$, providing the same asymptotic constant.
\end{Remark}

\begin{Remark} Whereas the constant in \eqref{sids5sasawqeyassdud} depends only on the dimension $d$,
the estimate for finite $h$ given by  \eref{sids5sassddedawqeyassdud}  depends explicitly on the aspect ratios of the simplices of $\Omega_h$ (the uniform bound on the ratio between the outer and inner radii of those simplices).
\end{Remark}

\begin{proof}
Let $V_h$ be a subspace of $H^1_0(\Omega)$ with $[|\Omega|/h^d]$ elements. Let $(v_k)$ be a basis of $V_h$.

Let $v_k'$ be the functions
associated with the   basis elements  $v_k$
through the equation
\begin{equation}\label{harmonsefdfddcalwsarmZ}
    \left\{\begin{array}{ll}
   \Delta v_k'=-\diiv\left( a(x) \nabla v_k (x)\right) \; & \text{in } \Omega\\
    v_k'=0  & \mbox{on}\;  \partial \Omega\end{array}.\right.
\end{equation}
It follows from equation \eref{sidasasedewdsssddaud} of theorem \ref{sdjhskjdhskdhkhje} that the following {\it transfer equation} holds

\begin{equation}\label{sidssxsasawqeyassdud}
\sup_{f \in L^2(\Omega)} \inf_{v\in
V_h} \frac{ \|u-v\|_{a\f}}{  \|f\|_{L^2(\Omega)}}=\sup_{w \in H^2\cap H^1_0(\Omega)} \inf_{v'\in
V_h'} \frac{ \|\nabla w-\nabla v'\|_{(L^2(\Omega))^d}}{  \|\Delta w\|_{L^2(\Omega)}}
\end{equation}
where  for $f\in L^2(\Omega)$,  $u$ is the solution of \eref{scalarproblem0}.
Using the eigenfunctions $\Psi_k$ of the Laplace-Dirichlet operator,  we arrive at
\begin{equation}\label{sidssqaxsasawdqeyassdud}
\frac{ \|\nabla w-\nabla v'\|_{(L^2(\Omega))^d}^2}{  \|\Delta w\|_{L^2(\Omega)}2}= \frac{\sum_{k=1}^\infty \frac{1}{\lambda_k}(\Delta w-\Delta v',\Psi_k)^2 }{\sum_{k=1}^\infty (\Delta w,\Psi_k)^2 }.
\end{equation}
When the supremum is taken with respect to $w \in H^2\cap H^1_0(\Omega)$, the right hand side of \eref{sidssqaxsasawdqeyassdud} can be minimized by taking
$V_h'$ to be the linear span of the first $[|\Omega|/h^d]$ eigenfunctions of the Laplace-Dirichlet operator on $\Omega$, because with such a basis the first $N(h)$ coefficients of $\Delta w$ are canceled, i.e.
 \begin{equation}\label{sidssqaxsasawdqeyassdudll}
\inf_{v'\in V_h'}\frac{ \|\nabla w-\nabla v'\|_{(L^2(\Omega))^d}^2}{  \|\Delta w\|_{L^2(\Omega)}2}= \frac{\sum_{k=N(h)+1}^\infty \frac{1}{\lambda_k}(\Delta w,\Psi_k)^2 }{\sum_{k=1}^\infty (\Delta w,\Psi_k)^2 }
\end{equation}
 with $N(h)=[|\Omega|/h^d]$. Then
\begin{equation}\label{siaqeyassdud}
\inf_{V_h',\,dim(V_h')=N(h)}\sup_{w \in H^2\cap H^1_0(\Omega)} \inf_{v'\in
V_h'} \frac{ \|\nabla w-\nabla v'\|_{(L^2(\Omega))^d}^2}{  \|\Delta w\|_{L^2(\Omega)}^2}=\frac{1}{\lambda_{N(h)+1}}.
\end{equation}
This follows by noting that the right hand side of equation \eqref{sidssqaxsasawdqeyassdudll} is less
than or equal to $\frac{1}{\lambda_{N(h)+1}}$ and that equality is obtained for $w = \Psi_{N(h)+1}$.

The optimality of the constant in $V_h'$  translates into the optimality of the constant associated with $V_h$ using the transfer equation \eref{sidssxsasawqeyassdud}, i.e.
\begin{equation}\label{sidssxsasahjhjuyuuwqeyassdud}
\inf_{V_h,\,dim(V_h)=N(h)}\sup_{f \in L^2(\Omega)} \inf_{v\in
V_h} \frac{ \|u-v\|_{a\f}}{  \|f\|_{L^2(\Omega)}}=\frac{1}{\sqrt{\lambda_{N(h)+1}}}
\end{equation}

We obtain the constant in \eref{sids5sasawqeyassdud} by using Weyl's
asymptotic formula for the eigenvalues of the Laplace-Dirichlet operator
on $\Omega$ \cite{MR1511560}.
\begin{equation}\label{sssddesqwsdedud}
\begin{split}
\lambda_k \sim  4\pi
\Big(\frac{\Gamma(1+\frac{d}{2})k}{|\Omega|}\Big)^\frac{2}{d},
\end{split}
\end{equation}
In equation \eref{sssddesqwsdedud}, $|\Omega|$ is the volume of $\Omega$, $d$ is the dimension of
the physical space and $\Gamma$ is the Gamma function defined by
$\Gamma(z):=\int_0^\infty t^{z-1}e^{-t}\,dt$.
 It follows from equation \eref{harmonsefdfddcalwsarmZ} that by defining $V_h=\Theta_h$ one obtain the smallest asymptotic constant in the right hand side of \eref{sids5sasawqeyassdud}. This being said, it should be recalled that the space $\Theta_h$ is not the unique space achieving this optimal constant \cite{Pinkus85}.

For the sake of clarity, an alternate (but similar) proof is provided below.
By proposition \ref{Csorsedfesol4}
\begin{equation}\label{sidasawqeyadsedsssddaud9}
\sup_{f \in L^2(\Omega)} \inf_{v\in V}  \frac{ \|u - v\|_{a\f}}{
\|f\|_{L^2(\Omega)}}= \sup_{z \in (\diiv a \nabla V)^{\perp}}\frac{\|
z\|_{L^2(\Omega)}}{\|\nabla z\|_{(L^2(\Omega))^d}},
\end{equation}

Taking $inf$ of both sides, we have
\begin{equation}\label{sidasawqeyadsedsssddaud99}
\inf_{V,\text{dim}(V)=N(h)}\sup_{f \in L^2(\Omega)} \inf_{v\in V}  \frac{ \|u - v\|_{a\f}}{
\|f\|_{L^2(\Omega)}}= \inf_{V_h,\text{dim}(V_h)=N(h)}\sup_{z \in (\diiv a \nabla V)^{\perp}}\frac{\|
z\|_{L^2(\Omega)}}{\|\nabla z\|_{(L^2(\Omega))^d}},
\end{equation}

Notice that the right hand side is the inverse of Rayleigh quotient, and $(\diiv a \nabla V)^{\perp}$ is a co-dimension $N(h)$ space, then by the Courant-Fischer min-max principle for the eigenvalues, we have
\begin{equation}\label{sidasawqeyadsedsssddaud999}
\inf_{V,\text{dim}(V)=N(h)}\sup_{f \in L^2(\Omega)} \inf_{v\in V}  \frac{ \|u - v\|_{a\f}}{
\|f\|_{L^2(\Omega)}}= \frac{1}{\sqrt{\lambda_{N(h)+1}}}
\end{equation}

Taking $V$ to be $\Theta_h$, then the optimal constant can be achieved asymptotically as $h\to 0$.
\end{proof}
\begin{Remark}
Theorem \ref{Corol1} is related to Melenk's $n$-widths analysis for elliptic problems \cite{MR1766938} where subsets of $L^2$ such as  piecewise discontinuous polynomials have been used  as an approximation basis. The main difference between \cite{MR1766938} and this section lies in the introduction of and the emphasis on the flux-norm ($\|.\|_{a\f}$) with respect to which errors become   independent of the contrast of the coefficients and the regularity of $a$. Moreover, this  allows us to obtain an {\it explicit} and optimal  constant in the  rate of convergence.
\end{Remark}

\begin{Remark}
Write
\begin{equation}\label{sids5sasasesdssdud}
\|g\|_{H^{-\nu}(\Omega)}^2:=\sum_{k=1}^\infty \frac{1}{\lambda_k^\nu}\big(g,\Psi_k\big)^2
\end{equation}
Then the space $\Theta_h$ also satisfies, for $\nu \in [0,1)$,
\begin{equation}\label{sids5sasawqeyasddsdssdud}
\lim_{h\rightarrow 0}\sup_{g \in H^{-\nu}(\Omega)} \inf_{v\in
\Theta_h} \frac{ \|u-v\|_{a\f}}{ h^{1-\nu} \|g\|_{H^{-\nu}(\Omega)}}=
\Big(\frac{1}{2\sqrt{\pi}}
\Big(\frac{1}{\Gamma(1+\frac{d}{2})}\Big)^\frac{1}{d}\Big)^{1-\nu}.
\end{equation}
\end{Remark}

\subsection{Vectorial elasticity equations.}

Let $(e_1,\ldots,e_d)$ be an orthonormal basis of $\R^d$. For $j\in
\{1,\ldots,d\}$ and $k\in \N^*=\{1,2, \cdots\}$, let $\tau_k^j$ be the solution of
\begin{equation}\label{harmonicscalwsar}
    \left\{\begin{array}{ll}
   -\diiv\left( C:\varepsilon( \tau_k^j) \right) = e_j \lambda_k \Psi_k,\; & \text{in } \Omega,\\
    \tau_k^j=0,  & \mbox{on}\;  \partial \Omega,\end{array}\right.
\end{equation}
where  $\Psi_k$ are the eigenfunctions \eref{kjskske23} of the \underline{scalar}
Laplace-Dirichlet operator in $\Omega$.  Let
$M:=\big[|\Omega|/h^d\big]$ be the integer part of $|\Omega|/h^d$
and $T_h$ be the linear space spanned by $\tau_k^j$ for $k\in
\{1,\ldots,M\}$ and $j\in \{1,\ldots,d\}$.

\begin{Remark}
Eigenmodes from a vector Laplace operator work as well. We use the eigenfunctions for a scalar Laplace operator because they are, in principle, simpler to compute and because our proof uses Weyl's asymptotic formula for the eigenvalues of the scalar Laplace-Dirichlet operator in order to obtain
the optimal constant in the right hand side of \eref{sids5sasawqebyassdud} and \eref{sidasaeddeedsssddadduddd999}. Also, the  eigenfunctions for the scalar Laplace operator encode information about the geometry of the domain $\Omega$.
\end{Remark}

\begin{Theorem}\label{Corol1del}
For $b \in (L^2(\Omega))^d$ let $u$ be the solution of
\eref{elasticity}. Then,
\begin{itemize}
\item
\begin{equation}\label{sids5sasawqebyassdud}
\lim_{h\rightarrow 0} \sup_{b \in (L^2(\Omega))^d} \inf_{v \in T_h}
\frac{ \|u-v\|_{C\f}}{h \|b\|_{(L^2(\Omega))^d}}=
\frac{1}{2\sqrt{\pi}}
\Big(\frac{1}{\Gamma(1+\frac{d}{2})}\Big)^\frac{1}{d}.
\end{equation}
Furthermore, the space $T_h$ leads (asymptotically) to the smallest possible constant in the right hand side of \eref{sids5sasawqebyassdud} among all subspaces of $H^1_0(\Omega)$ with $O(|\Omega/h^d|)$ elements.

\item
\begin{equation}\label{sidasaeddeedsssddadduddd999}
\inf_{V,\text{dim}(V)=N}\sup_{b \in (L^2(\Omega))^d} \inf_{v\in V}  \frac{ \|u - v\|_{C\f}}{ \|b\|_{(L^2(\Omega))^d}}= \frac{1}{2 \sqrt{\pi}}
\Big(\frac{|\Omega|}{\Gamma(1+\frac{d}{2})N}\Big)^\frac{1}{d} \big(1+\epsilon(N)\big)
\end{equation}
where the infimum is taken with respect to all subspaces of $(H^1_0(\Omega))^d$ with $N$ elements and $\epsilon(N)$ is converging towards zero as $N\rightarrow \infty$.
\end{itemize}
\end{Theorem}

\begin{proof}
Theorem \ref{Corol1del} is a straightforward application of equation \eref{sidasasedzsddaud} of theorem \ref{sdjhskzzhdsdkhje} and Weyl's estimate \eref{sssddesqwsdedud} (the proof is similar to the scalar case).
\end{proof}

Defining $\varphi_k$ as in subsection \ref{khdkhdlkwhe}, for $j\in
\{1,\ldots,d\}$  let $\Phi_k^j$ be the solution of
\begin{equation}\label{harmonicsdsdssdscalwsar}
    \left\{\begin{array}{ll}
   -\diiv\left( C(x):\varepsilon( \Phi_k^j) \right) = e_j \Delta\varphi_k,\; & \text{in } \Omega,\\
    \Phi_k^j=0,  & \mbox{on}\;  \partial \Omega,\end{array}\right.
\end{equation}

Define
\begin{equation}\label{harmseeddcalwddesarjhuuyu}
W_h:=\operatorname{span}\{\Phi_k^j\},
\end{equation}

\begin{Theorem}\label{Corsdsdedwdol1del}
For $b \in (L^2(\Omega))^d$ let $u$ be the solution of
\eref{elasticity}. Then,
\begin{equation}\label{sids5sasawqebyuyuyyassdud}
 \sup_{b \in (L^2(\Omega))^d} \inf_{v \in W_h}
\frac{ \|u-v\|_{C\f}}{ \|b\|_{(L^2(\Omega))^d}}\leq K h.
\end{equation}
where $K$ depends only on $\Omega$ and the aspect ratios of the simplices of $\Omega_h$.
\end{Theorem}
\begin{proof}
Theorem \ref{Corsdsdedwdol1del} is a straightforward application of equation \eref{sidasasedzsddaud} and the fact that one can approximate $H^2$ functions in the $H^1$ norm by functions from $\mathcal{L}_0^h$ with $\mathcal{O}(h)$ accuracy.
\end{proof}

\begin{Corollary}\label{Corsdxsddsol1del}
For $b \in (L^2(\Omega))^d$ let $u$ be the solution of
\eref{elasticity} and $u_h$ the finite element solution of \eref{elasticity} in $W_h$. Then,
\begin{equation}\label{sids5qssdud}
 \sup_{b \in (L^2(\Omega))^d} \inf_{v \in W_h}
\frac{ \|u-u_h\|_{(H^1_0(\Omega))^d}}{ \|b\|_{(L^2(\Omega))^d}}\leq \frac{K}{\lambda_{\min}(C)} h.
\end{equation}
where $K$ depends only on $\Omega$ and the aspect ratios of the simplices of $\Omega_h$.
\end{Corollary}

\begin{proof}
Corollary \ref{Corsdxsddsol1del} is a straightforward application of theorem \ref{Corsdsdedwdol1del}, inequality \eref{skkzshg} and Korn's inequality \eref{sksekdaseshg}.
\end{proof}

\section{A new class of inequalities}\label{inequalities}
The flux-norm (and harmonic coordinates in the scalar case \cite{MR2292954}) can be used to map a given operator $\diiv(a\nabla)$ ($\diiv(C:\varepsilon(u)$ for elasticity)) onto another operator $\diiv(a'\nabla)$ ($\diiv(C':\varepsilon(u)$)). Among all elliptic operators, those with divergence-free coefficients (as defined below) play a very special role in the sense that they can be written in both a divergence-form and a non-divergence form. We introduce a new class of inequalities for these operators.
We show  that these inequalities  hold under Cordes type conditions on the coefficients and conjecture that they hold without these conditions.

These inequalities will be required to hold only for divergence-free
conductivities because, by using the flux-norm through the transfer property defined in section \ref{sec1} or
harmonic coordinates as in \cite{MR2292954} (for the scalar case), we can map
non-divergence free conductivities onto divergence-free
conductivities and hence deduce homogenization results on the former
from inequalities on the latter.

\subsection{Scalar case.}\label{jhsdgkjgewh3}
 Let $a$ be the conductivity matrix
associated with equation \eref{scalarproblem0}. In this subsection,
we will assume that $a$ is uniformly elliptic, with bounded entries
and divergence free--i.e., for all $l\in \R^d$, $\diiv(a.l)=0$ (that is each column of $a$ is div free);
alternatively, for all $\varphi \in C^\infty_0(\Omega)$
\begin{equation}\label{kshelkejhe34}
\int_{\Omega}\nabla \varphi.a.l=0.
\end{equation}

Assume that $\Omega$ is a bounded domain in $\R^d$. For a
$d\times d$ matrix $M$, define
\begin{equation}
\Hess :M :=\sum_{i,j=1}^d \partial_i \partial_j M_{i,j}.
\end{equation}
We will also denote by $\Delta^{-1}M$ the $d\times d$ matrix defined by
\begin{equation}
(\Delta^{-1}M)_{i,j}=\Delta^{-1}M_{i,j}.
\end{equation}

\begin{Theorem}\label{kshkheie}
Let $a$ be a divergence free conductivity matrix. Then, the following
statements are equivalent for the same constant $C$:
\begin{itemize}
\item There exists $C>0$ such that for all $u\in H^1_0(\Omega)$,
\begin{equation}\label{kssjied}
\|u\|_{L^2(\Omega)}\leq C \big\|\Delta^{-1}\diiv(a\nabla
u)\big\|_{L^2(\Omega)}.
\end{equation}

\item There exists $C>0$ such that for all $u\in H^1_0(\Omega)$,
\begin{equation}\label{inequality1}
\big\|(\diiv(a\nabla))^{-1}\Delta u\big\|_{L^2(\Omega)}\leq C
\big\|u\big\|_{L^2(\Omega)}.
\end{equation}
\item  Writing $\theta_i$ the solutions of \eref{harmonsedcalwsar}. For all $(U_1,U_2,\ldots)\in \R^{\N^*}$,
\begin{equation}\label{skalkjehlkqjh3}
\big\|\sum_{i=1}^\infty U_i \theta_i \big\|^2_{L^2(\Omega)}\leq C^2
\sum_{i=1}^\infty U_i^2.
\end{equation}
\item The inverse of the operator
$-\diiv(a\nabla )$ (with Dirichlet boundary conditions) is a
continuous and bounded operators from $H^{-2}$ onto $L^2$. Moreover,
for $u\in H^{-2}(\Omega)$,
\begin{equation}\label{skksjhsewajjwwsseweee}
\big\|(\diiv a \nabla)^{-1} u\big\|_{L^2(\Omega)} \leq
C\|\Delta^{-1} u\|_{L^2(\Omega)}.
\end{equation}
\item There exists $C>0$ such that for all $u\in H^1_0(\Omega)$,
\begin{equation}\label{sisdswe33seeddqaww3dA}
\|u\|_{L^2(\Omega)}^2\leq C^2 \sum_{i=1}^\infty \left<\diiv(a\nabla
\frac{\Psi_i}{\lambda_i}), u\right>^2_{H^{-1}, H^{1}_0}.
\end{equation}
\item There exists $C>0$ such that
\begin{equation}\label{sissdswe33sedeeddqaww3dA}
\frac{1}{C}\leq \inf_{u\in H^1_0(\Omega)}\sup_{z\in H^2(\Omega)\cap
H^1_0(\Omega) }\frac{(\nabla z,a\nabla
u)_{L^2(\Omega)}}{\|u\|_{L^2(\Omega)} \|\Delta z\|_{L^2(\Omega)}}.
\end{equation}
\item There exists $C>0$ such that for all $u\in H^1_0(\Omega)$,
\begin{equation}\label{kssjsedied}
\|u\|_{L^2(\Omega)}\leq C \big\|\Delta^{-1}\Hess:(a
u)\big\|_{L^2(\Omega)}.
\end{equation}
\item There exists $C>0$ such that for all $u\in H^1_0(\Omega)$,
\begin{equation}\label{kssjsedseied}
\|u\|_{L^2(\Omega)}\leq C \big\|\Hess:(\Delta^{-1}(a
u))\big\|_{L^2(\Omega)}.
\end{equation}
\end{itemize}
\end{Theorem}
\begin{Remark}
Theorem \ref{kshkheie} can be related to the work of Conca
and Vanninathan \cite{MR1838500}, on uniform $H^2$-estimates in periodic homogenization, which established a similar result in the periodic
homogenization setting.
\end{Remark}

\begin{proof}
Let $U_j\in \R$.
 Observe that
\begin{equation}
-\diiv\big(a\nabla\sum_{j=1}^\infty \theta_j U_j)=\sum_{j=1}^\infty
\Psi_j \lambda_j U_j,
\end{equation}
hence
\begin{equation}
-\Delta^{-1}\diiv\big(a\nabla\sum_{j=1}^\infty \theta_j
U_j)=\sum_{j=1}^\infty \Psi_j  U_j.
\end{equation}
Identifying $u$ with $\sum_{j=1}^\infty \theta_j U_j$, it follows
that
\begin{equation}\label{hgfjhffhjgf}
\inf_{u \in L^2(\Omega)} \frac{\big\|\Delta^{-1}\diiv(a\nabla
u)\big\|_{L^2(\Omega)} }{\|u\|_{L^2(\Omega)}}\geq \frac{1}{C}
\end{equation}
is equivalent to
\begin{equation}\label{skksjhee}
\big\|\sum_{j=1}^\infty \theta_j U_j\big\|_{L^2(\Omega)}^2 \leq C^2
 U^T U.
\end{equation}
Observe that equation \eref{kssjied} is also equivalent to
\begin{equation}\label{skksjhsewwsseweee}
\big\|(\diiv a \nabla)^{-1} u\big\|_{L^2(\Omega)} \leq
C\|\Delta^{-1} u\|_{L^2(\Omega)},
\end{equation}
which is equivalent to the fact that the inverse of the operator
$-\diiv(a\nabla )$ (with Dirichlet boundary conditions) is a
continuous and bounded operator from $H^{-2}$ onto $L^2$. Finally,
the equivalence with \eref{sisdswe33seeddqaww3dA} is a consequence
of equation \eref{skalkjehlkqjh3}.
Let us now prove the equivalence with equations \eref{kssjsedied}
and \eref{kssjsedseied}. Observe that if $a$ is a divergence free
$d\times d$ symmetric matrix and $u \in H^2(\Omega)\cap
H^1_0(\Omega)$, then
\begin{equation}
\diiv(a\nabla u)=\Hess:(a u),
\end{equation}
since
\begin{equation}
\Hess:(a u)=\sum_{i,j=1}^d a_{i,j}\partial_i \partial_j
u+\sum_{j=1}^d \sum_{i=1}^d \partial_i a_{i,j} \partial_j u+
\sum_{i=1}^d \sum_{j=1}^d \partial_j a_{i,j} \partial_i u,
\end{equation}
$\sum_{i=1}^d \partial_i a_{i,j}=0$ and $\sum_{j=1}^d
\partial_j a_{i,j}=0$. It follows that
\begin{equation}
\Delta^{-1} \diiv(a\nabla u)= \Delta^{-1} \Hess: (au)= \Hess:
\Delta^{-1} (au),
\end{equation}
which concludes the proof of the equivalence between the statements.
\end{proof}

\begin{Theorem}\label{Prop6s}
If $a$ is divergence-free, then the statements of theorem
\ref{kshkheie} are implied by the following equivalent statements
with the same constant $C$.

\begin{itemize}
\item For all $u \in H^1_0(\Omega)\cap H^2(\Omega)$,
\begin{equation}\label{ksssseseaedsesedfdL}
\|\Delta u\|_{L^2(\Omega)} \leq C \|a:\Hess( u)\|_{L^2(\Omega)}.
\end{equation}
\item There exists $C>0$ such that for $u\in C^\infty_0(\Omega)$,

\begin{equation}\label{sissdswsed3sedeedaddsfsdeqasesedww3dA}
\|k^2 \F(u)\|_{L^2} \leq C \|k^T.\F(au) .k\|_{L^2},
\end{equation}
where $\F(u)$ is the Fourier transform of $u$.
\end{itemize}
 \end{Theorem}

 \begin{Remark}
 Concerning equation \eref{sissdswsed3sedeedaddsfsdeqasesedww3dA} since $u$ is compactly supported in $\Omega$, $u$  can be
 extended by zero outside of $\Omega$ without creating a Dirac part on its Hessian and $\F(u)$ is the Fourier transform of this extension.
 \end{Remark}

 \begin{proof}
Equation \eref{kssjsedied} is equivalent to
\begin{equation}\label{kssjsedaqeaediedL}
\frac{1}{C}\leq \inf_{u\in H^1_0(\Omega)} \sup_{\varphi \in
L^2(\Omega)} \frac{\big(\varphi,\Delta^{-1}\Hess:(a
u)\big)_{L^2(\Omega)}}{\|u\|_{L^2(\Omega)}
\|\varphi\|_{L^2(\Omega)}}.
\end{equation}
Denoting by $\psi$ the solution of $\Delta \psi=\varphi$ in
$H^1_0(\Omega)\cap H^2(\Omega)$, we obtain that
\eref{kssjsedaqeaediedL} is equivalent to

\begin{equation}\label{kssjsedssessedfdL}
\frac{1}{C}\leq \inf_{u\in H^1_0(\Omega)} \sup_{\psi \in
H^1_0(\Omega)\cap H^2(\Omega)} \frac{\big(\psi,\Hess:(a
u)\big)_{L^2(\Omega)}}{\|u\|_{L^2(\Omega)} \|\Delta
\psi\|_{L^2(\Omega)}}.
\end{equation}
Integrating by parts, we obtain that \eref{kssjsedssessedfdL} is
equivalent to
\begin{equation}\label{kssjseseddssedsesedfdL}
\frac{1}{C}\leq \inf_{u\in H^1_0(\Omega)} \sup_{\psi \in
H^1_0(\Omega)\cap H^2(\Omega)}
\frac{\big(a:\Hess(\psi),u\big)_{L^2(\Omega)}}{\|u\|_{L^2(\Omega)}
\|\Delta \psi\|_{L^2(\Omega)}}.
\end{equation}
Since $a$ is divergence free, $a:\Hess=\diiv(a\nabla .)$ and so there
exists $\psi$ such that $a:\Hess(\psi)=u$ with Dirichlet boundary
conditions. For such a $\psi$, we have
\begin{equation}\label{ksssedfsedsesedfdL}
\frac{\big(a:\Hess(\psi),u\big)_{L^2(\Omega)}}{\|u\|_{L^2(\Omega)}
\|\Delta \psi\|_{L^2(\Omega)}}=\frac{\|a:\Hess(
\psi)\|_{L^2(\Omega)}}{\|\Delta \psi\|_{L^2(\Omega)}}.
\end{equation}
It follows that inequality \eref{kssjseseddssedsesedfdL} is implied
by the inequality
\begin{equation}\label{kssssseesedsesedfdL}
\frac{1}{C}\leq \inf_{\psi \in H^1_0(\Omega)\cap H^2(\Omega)}
\frac{\|a:\Hess( \psi)\|_{L^2(\Omega)}}{\|\Delta
\psi\|_{L^2(\Omega)}}.
\end{equation}
The equivalence with \eref{sissdswsed3sedeedaddsfsdeqasesedww3dA}
follows from $a:\Hess(u)=\Hess:(au)$ and the conservation of the
$L^2$-norm by the Fourier transform.
 \end{proof}

\begin{Theorem}\label{swsmdblbhwsjhw3}
Let $a$ be a divergence free conductivity matrix.
\begin{itemize}
\item If $d=1$, then the statements of theorem \ref{Prop6s} are
true.
\item If $d=2$ and $\Omega$ is convex then  the statements of theorem \ref{Prop6s} are
true.
\item If $d\geq 3$, $\Omega$ is convex and the following Cordes
condition is satisfied
\begin{equation}\label{sshgdswsdsded7641}
\esssup_{x\in \Omega}\Big(
d-\frac{\big(\Tr[a(x)]\big)^2}{\Tr[{a^T(x)a(x)}]}\Big)<1
\end{equation}
then the  statements of theorem \ref{Prop6s} are true.
\item If $d\geq 2$, $\Omega$ is non-convex then there exists $C_\Omega>0$ such that if the following Cordes
condition is satisfied
\begin{equation}\label{sshgdswsdsdsedesed7641}
\esssup_{x\in \Omega}\Big(
d-\frac{\big(\Tr[a(x)]\big)^2}{\Tr[{a^T(x)a(x)}]}\Big)<C_\Omega
\end{equation}
then the  statements of theorem \ref{Prop6s} are true.
\end{itemize}
\end{Theorem}
\begin{proof}
In dimension one, if $a$ is divergence free then it is a constant
and the  statements of theorem \ref{Prop6s}  are trivially true.
Define
\begin{equation}\label{sshgdswsdsde3ed7641}
\beta_a:=\esssup_{x\in \Omega}\Big(
d-\frac{\big(\Tr[a(x)]\big)^2}{\Tr[{a^T(x)a(x)}]}\Big)
\end{equation}
Theorem  1.2.1 of
 \cite{MPG00} implies that if $\Omega$ is convex and
 $\beta_a<1$, then inequality \eref{ksssseseaedsesedfdL} is true. In
 dimension $2$, if $a$ is uniformly elliptic and bounded, then
 $\beta_a<1$. It follows that
if $d=2$ and $\Omega$ is convex or if $d\geq 3$, $\Omega$ is convex,
and $\beta_a<1$, then the statements of theorem \ref{Prop6s} are
true. The last statement of theorem \ref{swsmdblbhwsjhw3} is a
direct consequence of corollary 4.1 of \cite{MR1903306}.

For the sake of completeness we will include the proof of three bullet points here ($\Omega$ convex).
Write $\mathcal{L}$ the differential operator from $H^2(\Omega)$
onto $L^2(\Omega)$ defined by:
\begin{equation}
\mathcal{L}u:=\sum_{i,j} a_{ij} \partial_i \partial_j u
\end{equation}
Let us consider the equation
\begin{equation}\label{kljdwedddlkwjewlkjer}
\begin{cases}
\mathcal{L}u=f\quad \text{in}\quad \Omega\\
u=0\quad \text{on}\quad \partial \Omega
\end{cases}
\end{equation}
The following lemma corresponds to theorem 1.2.1 of
\cite{MPG00} (and $a$ does not need to be divergence free for the
validity of the following theorem). For the convenience of the reader, we will recall its proof in subsection \ref{ell1} of the appendix.

\begin{Lemma}\label{kahldssddkguelwhu}
Assume $\Omega$ to be convex with $C^2$-boundary. If $\beta_a<1$
then \eref{kljdwedddlkwjewlkjer} has a unique solution and
\begin{equation}
\|u\|_{H^2\cap H^1_0(\Omega)}\leq \frac{\esssup_{\Omega}\alpha(x)}{1-\sqrt{\beta_a}} \|f\|_{L^2(\Omega)}
\end{equation}
where $\alpha(x):=(\Sigma_{i=1}^d a_{ii}(x))/\sum_{i,j=1}^d (a_{ij}(x))^2$
\end{Lemma}

$\beta_a$ is a measure of the anisotropy of $a$. In particular, for
the identity matrix one has $\beta_{I_d}=0$. Furthermore in dimension $2$
\begin{equation}
\beta_a=1-\operatorname{essinf}_{x\in \Omega} \frac{2\lambda_{\min}(a(x))\lambda_{\max}(a(x))}{(\lambda_{\min}(a(x)))^2+(\lambda_{\max}(a(x)))^2}
\end{equation}
and one always have $\beta_a<1$ provided that $a$ is uniformly elliptic and bounded. The first three bullet points of theorem \ref{swsmdblbhwsjhw3} follow by observing that if $\beta_a<1$ then
\begin{equation}
\|u\|_{H^2\cap H^1_0(\Omega)}\leq C \|\sum_{i,j} a_{ij}\partial_i \partial_j u\|_{L^2(\Omega)}
\end{equation}
which implies inequality  \eref{ksssseseaedsesedfdL}.

\end{proof}

\subsubsection{A brief reminder on the mapping using harmonic coordinates.}\label{reminder}
Consider the divergence-form elliptic scalar problem
\eref{scalarproblem0}. Let $F$ denote the harmonic coordinates associated
with \eref{scalarproblem0}--i.e., $F(x)=\big(F_1(x),\ldots,F_d(x)\big)$ is a $d$-dimensional vector field
whose entries satisfy
\begin{equation}\label{harmonic}
\begin{cases}
\diiv a \nabla F_i=0 \quad \text{in}\quad \Omega\\
F_i(x)=x_i \quad \text{on}\quad \partial \Omega.
\end{cases}
\end{equation}

It is easy to show that $F$ is a mapping from $\Omega$ onto
$\Omega$. In dimension one, $F$ is trivially a homeomorphism. In
dimension two, this property still holds for convex  domains
\cite{MR2001070, MR1892102}. In dimensions three and higher, $F$ may
be non-injective (even if $a$ is smooth, we refer to
\cite{MR1892102}, \cite{MR2073507}).

Define $Q$ to be the
positive symmetric $d\times d$ matrix defined by
\begin{equation}\label{kqkjhlkwjhl3}
Q:=\frac{(\nabla F)^T a \nabla F}{\det{\nabla F}}\circ F^{-1}.
\end{equation}

It is shown in \cite{MR2292954} that $Q$ is divergence free.
Moreover, writing $u$ the solution of \eref{scalarproblem0} and $\|u\|_a:=\int_\Omega \nabla u\cdot a\nabla u$ one has for $v\in H^1_0(\Omega)$
\begin{equation}\label{jgjgkjghj}
\|u-v\|_a=\|\hat{u}-\hat{v}\|_Q,
\end{equation}
where $\hat{v}:=v\circ F^{-1}$ and $\hat{u}:=u\circ F^{-1}$ solves
\begin{equation}
-\sum_{i,j} Q_{i,j}\partial_i \partial_j \hat{u}=\frac{g}{\det(\nabla F)}\circ F^{-1}
\end{equation}
Note that \eqref{jgjgkjghj} allows one to transfer the error for a general conductivity matrix $a$ to a special divergence-free conductivity matrix $Q$. Observe that the energy norm was used in  \cite{MR2292954} (and \eref{jgjgkjghj}, instead of the flux norm) under  bounded contrast assumptions
on $a$.

The approximation results obtained in \cite{MR2292954} are based on \eref{jgjgkjghj} and  can also be derived by using the new class of inequalities described above for $Q$.

\subsection{Tensorial case.}\label{NewIneq2}
 Let $C$ be the elastic stiffness matrix
associated with equation \eref{elasticity}. In this subsection, we
will assume that $C$ is uniformly elliptic, has bounded entries and
is divergence free--i.e., $C$ is such that for all $l\in \R^{d\times d}$,
$\diiv(C:l)=0$; alternatively, for all $\varphi \in (C^\infty_0(\Omega))^d$,
\begin{equation}\label{kshelkejhe3se4}
\int_{\Omega}(\nabla \varphi)^T:C:l=0.
\end{equation}

The inequalities given below will allow us to deduce homogenization
results for arbitrary elasticity tensors (not necessarily
divergence-free) by using harmonic displacements and the flux-norm to map non-divergence
free tensors onto divergence-free tensors.

For a $d\times d\times d$ tensor $M$, denote by $\Hess :M$ the vector
\begin{equation}
(\Hess :M)_k :=\sum_{i,j=1}^d \partial_i \partial_j M_{i,j,k}.
\end{equation}
Let $\Delta^{-1}M$ denote the $d\times d\times d$ tensor
defined by
\begin{equation}
(\Delta^{-1}M)_{i,j,k}=\Delta^{-1}M_{i,j,k}.
\end{equation}

The proof of the following theorem is almost identical to the proof
of theorem \ref{kshkheie}.
\begin{Theorem}\label{kshksededsheie}
Let $C$ be a divergence free elasticity tensor. The
following statements are equivalent for the same constant $\gamma$:
\begin{itemize}
\item There exists $\gamma>0$ such that for all $u\in (H^1_0(\Omega))^d$,
\begin{equation}\label{ksxssjied}
\|u\|_{(L^2(\Omega))^d}\leq \gamma \big\|\Delta^{-1}\diiv(C:\varepsilon(
u))\big\|_{(L^2(\Omega))^d}.
\end{equation}
\item There exists $\gamma>0$ such that for all $u\in (H^1_0(\Omega))^d$,
\begin{equation}
\big\|(\diiv(C:\varepsilon( .)))^{-1}\Delta u\big\|_{(L^2(\Omega))^d}\leq
\gamma \big\|u\big\|_{(L^2(\Omega))^d}.
\end{equation}
\item  For all $(U_1,U_2,\ldots)\in (\R^d)^{\N^*}$,
\begin{equation}\label{skalwedwedksjehselkqjh3}
\big\|\sum_{k=1}^\infty \sum_{j=1}^d U_k^j \tau_k^j
\big\|^2_{L^2(\Omega)}\leq \gamma^2 \sum_{k=1}^\infty U_k^2,
\end{equation}
where $\{\tau_k^j\}$ is the  basis defined in
\eref{harmonicscalwsar}.
\item The inverse of the operator
$-\diiv(C:\varepsilon(.) )$ (with Dirichlet boundary conditions) is a
continuous and bounded operator from $(H^{-2})^d$ onto $(L^2)^d$.
Moreover, for $u\in (H^{-2}(\Omega))^d$,
\begin{equation}\label{skksjhsewadejjwwsseweee}
\big\|(\diiv C: \varepsilon(.))^{-1} u\big\|_{(L^2(\Omega))^d} \leq \gamma
\|\Delta^{-1} u\|_{(L^2(\Omega))^d}.
\end{equation}
\item There exists $\gamma>0$ such that for all $u\in (H^1_0(\Omega))^d$,
\begin{equation}\label{sisdswe33seesdrfdzsdsqaww3dA}
\|u\|_{(L^2(\Omega))^d}^2\leq \gamma^2 \sum_{i=1}^\infty
\sum_{j=1}^d \left<(\diiv\big(C: (
\frac{\nabla\Psi_i}{\lambda_i}\otimes e_j)\big),
u\right>^2_{(H^{-1},H^1)}.
\end{equation}

\item There exists $\gamma>0$ such that
\begin{equation}\label{sissdswe33zdsdsedeeddqaww3dA}
\frac{1}{\gamma}\leq \inf_{u\in (H^1_0(\Omega))^d}\sup_{z\in
(H^2(\Omega)\cap H^1_0(\Omega))^d }\frac{((\nabla z)^T:C:\varepsilon(
u))_{L^2(\Omega)}}{\|u\|_{(L^2(\Omega))^d} \|\Delta
z\|_{(L^2(\Omega))^d}}.
\end{equation}
\item There exists $\gamma>0$ such that for all $u\in (H^1_0(\Omega))^d$,
\begin{equation}\label{kssjsskesekwkedied}
\|u\|_{(L^2(\Omega))^d}\leq \gamma
\big\|\Delta^{-1}\Hess:(u.C)\big\|_{L^2(\Omega)}.
\end{equation}
\item There exists $\gamma>0$ such that for all $u\in (H^1_0(\Omega))^d$,
\begin{equation}\label{kssjsedsdsdeseied}
\|u\|_{(L^2(\Omega))^d}\leq \gamma
\big\|\Hess:(\Delta^{-1}(u.C))\big\|_{(L^2(\Omega))^d}.
\end{equation}
\end{itemize}
\end{Theorem}

\begin{Theorem}\label{Prewkwop6s}
If $C$ is divergence-free, the statements of theorem
\ref{kshksededsheie} are implied by the following statement
 with the same constant $\gamma$.

\begin{itemize}
\item For all $u \in (H^1_0(\Omega)\cap H^2(\Omega))^d$,
\begin{equation}\label{ksssseseaedsessedfdL}
\|\Delta u\|_{(L^2(\Omega))^d} \leq \gamma
\|\Hess:(u.C)\|_{(L^2(\Omega))^d}.
\end{equation}
\end{itemize}
 \end{Theorem}
 \begin{proof}
The proof is similar to that of theorem \ref{Prop6s}.
 \end{proof}

\subsubsection{A Cordes Condition for tensorial non-divergence form elliptic equations}

Let us now show that the inequality in theorem \ref{Prewkwop6s},
and hence the inequalities of theorem \ref{kshksededsheie}, are
satisfied if $C$ satisfies a Cordes type condition. The proof of the
following theorem is an adaptation of the proof of theorem 1.2.1 of
\cite{MPG00} (note that $C$ does not need to be divergence free in order for
the following theorem to be valid).

Let $\mathcal{L}$ denote the differential operator from $(H^2(\Omega)^d$
onto $(L^2(\Omega))^d$ defined by
\begin{equation}
(\mathcal{L}u)_j:=\sum_{i,k,l} C_{ijkl} \partial_i \partial_k u_l.
\end{equation}

Let us consider the equation
\begin{equation}\label{kljdlkwjewlkjer}
\begin{cases}
\mathcal{L}u=f\quad \text{in}\quad \Omega\\
u=0\quad \text{on}\quad \partial \Omega.
\end{cases}
\end{equation}

Let $B$ be the $d\times d$ matrix defined by $B_{jm}=\sum_{k=1}^d
C_{kmkj}$. Let $A$ be the $d\times d$ matrix defined by
$A_{j'm}=\sum_{i,k,l=1}^d C_{imkl} C_{ij'kl}$. Define
\begin{equation}\label{shjdghshkwjw}
\beta_{C}:= d^2 -\Tr[BA^{-1}B^T].
\end{equation}

\begin{Theorem}\label{kahlkguelwhu}
Assume $\Omega$ is convex with a $C^2$-boundary. If $\beta_C<1$,
then \eref{kljdlkwjewlkjer} has a unique solution and
\begin{equation}
\|u\|_{(H^2\cap H^1_0(\Omega))^d}\leq K \|f\|_{(L^2(\Omega))^d},
\end{equation}
where $K$ is a function of $\beta_C$ and $\|B
A^{-1}\|_{(L^\infty(\Omega))^{d\times d}}$.
\end{Theorem}

\begin{Remark}
$\beta_C$ is a measure of the anisotropy of $C$. In particular, for
the identity tensor, one has $\beta_{I_d}=0$.
\end{Remark}

\begin{proof}

Let $u$ be the solution of $\mathcal{L}u=f$ with Dirichlet boundary
conditions (assuming that it exists). Let $\alpha$ be a field of
$d\times d$ invertible matrices. Observe that \eref{kljdlkwjewlkjer}
is equivalent to
\begin{equation}
\Delta u=\alpha f+\Delta u-\alpha \mathcal{L}u.
\end{equation}

Consider the mapping $T:(H^2\cap H^1_0(\Omega))^d\rightarrow
(H^2\cap H^1_0(\Omega))^d$ defined by $v=Tw$, where $v$ be the unique
solution of the Dirichlet problem for Poisson equation
\begin{equation}
\Delta v=\alpha f+\Delta w-\alpha \mathcal{L}w.
\end{equation}

Let us now choose $\alpha$ so that $T$ is a contraction.

Note that
\begin{equation}
\begin{split}
\big\|T w_1-T w_2\big\|_{(H^2\cap
H^1_0(\Omega))^d}=\|v_1-v_2\|_{(H^2\cap H^1_0(\Omega))^d}.
\end{split}
\end{equation}
Using the convexity of $\Omega$, one obtains the following classical
inequality satisfied by the Laplace operator (see lemma 1.2.2 of
\cite{MPG00}):
\begin{equation}
\begin{split}
\|v_1-v_2\|_{(H^2\cap H^1_0(\Omega))^d}\leq \|\Delta
(v_1-v_2)\|_{(L^2(\Omega))^d}.
\end{split}
\end{equation}
Hence,
\begin{equation}
\begin{split}
\big\|T w_1-T w_2\big\|_{(H^2\cap H^1_0(\Omega))^d}^2\leq &\|\Delta
(w_1-w_2)-\alpha \mathcal{L}(w_1-w_2)\|_{(L^2(\Omega))^d}^2\\
= &  \Big\|\sum_{i,j,k,l=1}^d e_j
\big(\delta_{jl}\delta_{ki}-\sum_{j'=1}^d\alpha_{jj'} C_{ij'kl}\big)
\partial_i
\partial_k (w_1^l-w_2^l)\Big\|_{(L^2(\Omega))^d}^2.
\end{split}
\end{equation}
Using the Cauchy-Schwarz  inequality, we obtain that
\begin{equation}
\begin{split}
\big\|T w_1-T w_2\big\|_{(H^2\cap H^1_0(\Omega))^d}^2\leq &
\int_{\Omega}  \big(\sum_{i,j,k,l=1}^d
(\delta_{jl}\delta_{ki}-\sum_{j'=1}^d \alpha_{jj'} C_{ij'kl})^2\big)
 \\& \big(\sum_{i,k,l=1}^d
(\partial_i
\partial_k (w_1^l-w_2^l))^2 \big).
\end{split}
\end{equation}
Hence, writing
\begin{equation}
\beta_{\alpha,C}:=\sum_{i,j,k,l=1}^d
(\delta_{jl}\delta_{ki}-\sum_{j'=1}^d \alpha_{jj'} C_{ij'kl})^2,
\end{equation}
we obtain that
\begin{equation}
\begin{split}
\big\|T w_1-T w_2\big\|_{(H^2\cap H^1_0(\Omega))^d}^2\leq
\esssup_{x\in \Omega}\beta_{\alpha,C}(x) \big\|w_1-
w_2\big\|_{(H^2\cap H^1_0(\Omega))^d}^2.
\end{split}
\end{equation}
Observe that
\begin{equation}
\beta_{\alpha,C}:= d^2 - 2 \sum_{j',j,k=1}^d  \alpha_{jj'}
C_{kj'kj}+ \sum_{i,j,k,l=1}^d (\sum_{j'=1}^d \alpha_{jj'}
C_{ij'kl})^2.
\end{equation}
Taking variations with respect to $\alpha$, one must have, at the
minimum, that for all $j,m$,
\begin{equation}
 \sum_{i,k,l=1}^d C_{imkl}(\sum_{j'=1}^d \alpha_{jj'} C_{ij'kl})=  \sum_{k=1}^d
C_{kmkj}.
\end{equation}
Hence,
\begin{equation}\label{kwehwjhe}
 \sum_{j'=1}^d \alpha_{jj'} \sum_{i,k,l=1}^d C_{imkl} C_{ij'kl}=  \sum_{k=1}^d
C_{kmkj}.
\end{equation}
Let $B$ be the matrix defined by $B_{jm}=\sum_{k=1}^d C_{kmkj}$. Let
$A$ be the matrix defined by $A_{j'm}=\sum_{i,k,l=1}^d C_{imkl}
C_{ij'kl}$. Then \eref{kwehwjhe} can be written as
\begin{equation}
\alpha A=B,
\end{equation}
which leads to
\begin{equation}
\alpha^* =B A^{-1}.
\end{equation}
For such a choice, one has
\begin{equation}
\sum_{i,j,k,l=1}^d (\sum_{j'=1}^d \alpha_{jj'}^*
C_{ij'kl})^2=\sum_{j,m,k=1}^d  \alpha_{jm}^* C_{kmkj}.
\end{equation}
Hence, at the minimum, $\beta_{\alpha,C}=\beta_C$ with
\begin{equation}
\beta_{C}:= d^2 -\Tr[BA^{-1}B^T].
\end{equation}
For that specific choice of $\alpha$, if $\beta_C<1$, then $T$ is a
contraction and we obtain the existence and solution of
\eref{kljdlkwjewlkjer} through the fixed point theorem. Moreover,

\begin{equation}
\|\Delta u\|_{(L^2(\Omega))^d}\leq \|\alpha^*
f\|_{L^2(\Omega)}+\beta_C^\frac{1}{2} \|\Delta u\|_{(L^2(\Omega))^d},
\end{equation}
which concludes the proof.

\end{proof}

As a direct consequence of theorem \ref{Prewkwop6s} and theorem
\ref{kahlkguelwhu}, we obtain the following theorem.
\begin{Theorem}\label{swsmdblbhwsjhw3o}
Let $C$ be a divergence free bounded, uniformly elliptic, fourth
order tensor. Assume $\Omega$ is convex with a $C^2$-boundary. If
$\beta_C$, defined by \eref{shjdghshkwjw}, is strictly bounded from
above by one, then the inequalities of theorem \ref{Prewkwop6s}
and theorem \ref{kshksededsheie} are satisfied.
\end{Theorem}

\section{Application of the flux-norm to theoretical non-conforming Galerkin.}\label{lkjlsdhkj}
The change of coordinates used in \cite{MR2292954} (see also subsection \ref{reminder}) to obtain error estimates for finite element solutions of scalar equation \eref{scalarproblem0} in two-dimensions admits no straightforward  generalization for  vectorial elasticity equations. In this section, we show how the flux-norm can be used to obtain error estimates for theoretical discontinuous Galerkin solutions of \eref{scalarproblem0} and \eref{elasticity}. These estimates are based on the inequalities introduced in section \ref{inequalities} and the control of the non-conforming error associated with the theoretical discontinuous Galerkin method. The control of the non-conforming error could be implemented by methods such as the penalization method. Its analysis is, however, difficult in general and will not be done here.
In the scalar case, we refer to \cite{OwhZha07errata} for the control of the non-conforming error.

\subsection{Scalar equations} \label{scalarequations}
Let $w\in H^2\cap H^1_0(\Omega)$  such that $-\Delta w=f$. Let $u$ be the solution in $H^1_0(\Omega)$ of
\begin{equation}\label{scalarpsdsdssdedm0}
    \int_\Omega (\nabla \varphi)^T a   \nabla u =\int_{\Omega}(\nabla \varphi)^T \nabla w \quad \varphi \in H^1_0(\Omega)
\end{equation}

Let $\mathcal{V}$ be a finite dimensional linear subspace of
$(L^2(\Omega))^d$.

We write  $\zeta_{\V}$ an approximation of the gradient of the solution of
\eref{scalarproblem0} in $\V$ obtained by solving \eref{klwkjhelkjhde3}--i.e.,
$\zeta_{\V}$ is defined such that for all $\eta \in \V$,
\begin{equation}\label{klwkjhelkjhde3}
\int_{\Omega} \eta^T a \zeta_{\V}=\int_{\Omega}\eta^T \nabla w.
\end{equation}

For $\xi\in
(L^2(\Omega))^d$, denote by $\xi=\xi_{curl}+\xi_{pot}$  the Weyl-Helmholtz decomposition of $\xi$ (see Definition \ref{defWeyl}).
\begin{Definition}
Write
\begin{equation}\label{lkjedswewhhh3}
\mathcal{K}_\V:=\sup_{\zeta \in \V}
\frac{\|\zeta_{curl}\|_{(L^2(\Omega))^d}}{\|\zeta\|_{(L^2(\Omega))^d}}.
\end{equation}
\end{Definition}
$\mathcal{K}_\V$ is related to the ``non-conforming error'' associated with
$\V$ (see for instance \cite{BrSc02} chapter 10).
If $\mathcal{K}_\V >0$ then the space $\V$ must contain functions that are not exact gradients. Moreover, it determines the ``distance'' between $\V$ and $L^2_{pot}$ (see definition \ref{defWeyl}).

\begin{Definition}
Write
\begin{equation}\label{lkjedswddeewhhh3}
\mathcal{D}_{\V}:=\inf_{a',\V'\,:\, \diiv(a'\V')=\diiv(a\V)}\sup_{w'\in H^2(\Omega)\cap H^1_0(\Omega)} \inf_{\zeta' \in \V'} \frac{\big\|(a'(\nabla u'-\zeta'))_{pot}\big\|_{(L^2(\Omega))^d}}{\|\Delta w'\|_{L^2(\Omega)}}
\end{equation}
\end{Definition}
The first minimum in  \eref{lkjedswddeewhhh3} is taken with respect to all finite dimensional linear subspaces  $\V'$ of
$(L^2(\Omega))^d$, and all bounded uniformly elliptic matrices $a'$ ($a_{ij}'\in L^\infty(\Omega)$) such that $\diiv(a'\V')=\diiv(a\V)$. Furthermore, $u'$ in \eref{lkjedswddeewhhh3} is defined as the (weak) solution of $\diiv(a'\nabla u')=\Delta w'$ with Dirichlet boundary condition on $\partial \Omega$.
Due to Theorem \ref{sdjhskjdhskdhkhje}, the $\inf_{a',\V'\,:\, \diiv(a'\V')=\diiv(a\V)}$ can be dropped.  However, we keep it to emphasize the independence of  the choice of $\V'$ and $a'$ as long as they satisfy $\diiv(a'\V')=\diiv(a\V)$.

\begin{Theorem}\label{kjwhkejhwejd}
There exists a constant $C^*>0$ depending only on $\lambda_{\min}(a)$ and $\lambda_{\max}(a)$ such that for $\mathcal{K}_\V\leq C^*$,
\begin{equation}\label{jakhgsjahsd}
\begin{split}
\|\nabla u - \zeta_{\V}\|_{(L^2(\Omega))^d}\leq C \| f \|_{L^2(\Omega)} \big(\mathcal{D}_{\V}+\mathcal{K}_\V\big)
\end{split}
\end{equation}
where $u$ is the solution of \eref{scalarpsdsdssdedm0}, $\zeta_\V$ the solution of \eref{klwkjhelkjhde3} and
 $C$ is a constant depending only on $\lambda_{\min}(a)$ and $\lambda_{\max}(a)$.
\end{Theorem}
\begin{Remark}
Theorem \ref{kjwhkejhwejd} is in essence stating that the approximation error associated with $\V$ and the operator $\diiv(a\nabla)$ is proportional to  $\mathcal{D}_\V$ and $\mathcal{K}_\V$. $\mathcal{K}_\V$ is related to the non-conforming error associated to $\V$. $\mathcal{D}_\V$ is the minimum  (over $a'$, $\V'$ such that $\diiv(a'\V')=\diiv(a\V)$) approximation error associated to $\V'$ and the operator $\diiv(a'\nabla)$. Hence, $\mathcal{D}_{\V}$ and the transfer property allow us to equate the  accuracy of a scheme associated with  $\V'$ and a conductivity $a'$ to the accuracy of the scheme associated with  $\V$ and the conductivity $a$ provided that $\diiv(a'\V')=\diiv(a\V)$.
\end{Remark}
\begin{Remark}
In fact, it is possible to deduce from theorem  \ref{kkhkjhkjhu} that the maximum approximation error associated to   $\V$ and the operator $\diiv(a\nabla)$ can be bounded from below by a multiple of $\big(\mathcal{D}_{\V}+\mathcal{K}_\V\big)$ (see also equation (10.1.6) of \cite{BrSc02}).
\end{Remark}

\begin{Remark}
If the elements of $\V$ are of the form  $\eta=\sum_{\tau \in \Omega_h} 1_{(x\in \tau)} \nabla v$ where $v$ belongs to a linear space of functions with discontinuities at the boundaries of the simplices of $\Omega_h$ then
 we can replace  the right hand side of \eref{klwkjhelkjhde3} by $-\int_\Omega v \Delta w$ (see subsection 1.3 of \cite{MR2292954}). This modification doesn't affect the validity of \eref{jakhgsjahsd} since the difference between the two terms remains controlled by $\mathcal{D}_\V$.
 For  clarity of  presentation, we have used the formulation \eref{klwkjhelkjhde3}.
\end{Remark}

In order to prove theorem \ref{kjwhkejhwejd}, we will need the following lemma
\begin{Lemma}\label{kkhkjhkjhu}
There exists $C$ depending only on $\lambda_{\min}(a),\lambda_{\max}(a)$ such that for $u\in H^1_0(\Omega)$ and $\zeta \in (L^2(\Omega))^d$
\begin{equation}\label{dksjhldskhd}
\|\nabla u -\zeta\|_{(L^2(\Omega))^d}\leq C \Big(\big\|(a(\nabla u-\zeta))_{pot}\big\|_{(L^2(\Omega))^d}+\|\zeta_{curl}\|_{(L^2(\Omega))^d}\Big)
\end{equation}
\begin{equation}\label{jhdsdghsg}
\|\nabla u -\zeta\|_{(L^2(\Omega))^d}\geq \frac{1}{C} \Big(\big\|(a(\nabla u-\zeta))_{pot}\big\|_{(L^2(\Omega))^d}+\|\zeta_{curl}\|_{(L^2(\Omega))^d}\Big)
\end{equation}
\end{Lemma}
\begin{proof}
For the proof of \eref{jhdsdghsg}, observe that $$\|\nabla u -\zeta\|_{(L^2(\Omega))^d}=\|\nabla u -\zeta_{pot}\|_{(L^2(\Omega))^d}+\|\zeta_{curl}\|_{(L^2(\Omega))^d}.$$
Furthermore,
\begin{equation}\label{jhdsdghsddg}
\begin{split}
\big\|(a(\nabla u-\zeta))_{pot}\big\|_{(L^2(\Omega))^d}\leq& \big\|(a(\nabla u-\zeta_{pot}))_{pot}\big\|_{(L^2(\Omega))^d}+\big\|(a\zeta_{curl})_{pot}\big\|_{(L^2(\Omega))^d}\\
&\leq \lambda_{\max}(a) \big\|\nabla u-\zeta_{pot}\big\|_{(L^2(\Omega))^d}+ \lambda_{\max}(a) \big\|\zeta_{curl}\big\|_{(L^2(\Omega))^d}
\end{split}
\end{equation}
For \eref{dksjhldskhd}, observe that
\begin{equation}
\int_{\Omega} (\nabla u-\zeta)^T a (\nabla u-\zeta)=
\int_{\Omega} (\nabla u-\zeta_{pot})^T \big(a (\nabla u-\zeta)\big)_{pot}+
\int_{\Omega} \zeta_{curl}^T a (\nabla u-\zeta)
\end{equation}
It follows from Cauchy-Schwarz  inequality  that
\begin{equation}
\begin{split}
\lambda_{\min}(a)\|\nabla u -\zeta\|_{(L^2(\Omega))^d}\leq &
\frac{\|\nabla u -\zeta_{pot}\|_{(L^2(\Omega))^d}}{\|\nabla u -\zeta\|_{(L^2(\Omega))^d}} \big\|(a(\nabla u-\zeta))_{pot}\big\|_{(L^2(\Omega))^d}\\&+\lambda_{max}(a)\|\zeta_{curl}\|_{(L^2(\Omega))^d}
\end{split}
\end{equation}
\end{proof}
We also need the following lemma, which corresponds to lemma (10.1.1) of \cite{BrSc02}
\begin{Lemma}\label{kkwekjhekjhe}
Let $H$ be  a Hilbert space, $V$ and $V_h$ be subspaces of $H$ ($V_h$ may not be a subset of $V$). Assume that $a(.,.)$ is continuous bilinear form
on $H$ which is coercive on $V_h$, with respective continuity and coercivity constants $C$ and $\gamma$. Let $u\in V$ solve
\begin{equation}
a(u,v)=F(v)\quad \, \forall v\in V
\end{equation}
where $F\in H'$ ($H'$ is the dual of $H$). Let $u_h\in V_h$ solve
\begin{equation}
a(u_h,v)=F(v)\quad \, \forall v\in V_h
\end{equation}
Then
\begin{equation}\label{kskdjkdhkhjed}
\|u-u_h\|_H \leq \big(1+\frac{C}{\gamma}\big) \inf_{w\in V_h}\|u-w\|_{H}+\frac{1}{\gamma}\sup_{w\in V_h\setminus \{0\}}\frac{a(u-u_h,w)}{\|w\|_H}
\end{equation}

\end{Lemma}

We now proceed by proving theorem \ref{kjwhkejhwejd}.
\begin{proof}
Using lemma \ref{kkhkjhkjhu}, we obtain that
\begin{equation}
 \|\nabla u - \zeta\|_{(L^2(\Omega))^d}\leq C\Big(\big\|(a(\nabla u-\zeta))_{pot}\big\|_{(L^2(\Omega))^d}+\|\zeta\|_{(L^2(\Omega))^d} \mathcal{K}_\V\Big)
\end{equation}
Using  the triangle inequality $\|\zeta\|_{(L^2(\Omega))^d}\leq \|\nabla u-\zeta\|_{(L^2(\Omega))^d}+\|\nabla u\|_{(L^2(\Omega))^d}$,
we obtain that
\begin{equation}\label{jkshgshgdhe}
 \|\nabla u - \zeta\|_{(L^2(\Omega))^d}\leq \frac{C}{1-C\mathcal{K}_\V}\Big(\big\|(a(\nabla u-\zeta))_{pot}\big\|_{(L^2(\Omega))^d}+\|\nabla u\|_{(L^2(\Omega))^d} \mathcal{K}_\V\Big)
\end{equation}
from which we deduce that
\begin{equation}\label{sddejksdhgshgdhe}
\inf_{\zeta \in \V} \|\nabla u - \zeta\|_{(L^2(\Omega))^d}\leq \frac{C}{1-C\mathcal{K}_\V}\inf_{\zeta \in \V} \Big(\big\|(a(\nabla u-\zeta))_{pot}\big\|_{(L^2(\Omega))^d}+\|\nabla u\|_{(L^2(\Omega))^d} \mathcal{K}_\V\Big)
\end{equation}
($C^*$ in the statement of the theorem is chosen so that $\mathcal{K}_\V<C^*$ implies $C \mathcal{K}_\V<0.5$).
We obtain from lemma \ref{kkwekjhekjhe}  that (observe that the last term in equation \eref{kskdjkdhkhjed} is the non-conforming error and that it is bounded by $C\|\nabla u\|_{(L^2(\Omega))^d} \sup_{\zeta \in \V} \frac{\|\zeta_{curl}\|_{(L^2(\Omega))^d}}{\|\zeta\|_{(L^2(\Omega))^d}}$ for an appropriate constant $C$).
\begin{equation}\label{jsdhgsjdgjdgh}
\|\nabla u - \zeta \|_{(L^2(\Omega))^d}\leq C \Big(\inf_{\zeta \in \V} \|\nabla u - \zeta \|_{(L^2(\Omega))^d}+
\|\nabla u\|_{(L^2(\Omega))^d} \sup_{\zeta \in \V} \frac{\|\zeta_{curl}\|_{(L^2(\Omega))^d}}{\|\zeta\|_{(L^2(\Omega))^d}}\Big)
\end{equation}
Combining \eref{sddejksdhgshgdhe} with \eref{jsdhgsjdgjdgh}, we conclude using theorem \ref{sdjhskjdhskdhsdsdsdsxkhje} and the Poincar\'{e} inequality.
\end{proof}

Let us now show how theorem \ref{kjwhkejhwejd} can be combined with the new class of inequalities obtained in sub-section \ref{jhsdgkjgewh3} to obtain homogenization results for arbitrarily rough coefficients $a$. Let $M$ be a uniformly elliptic $d\times d$ matrix (observe that uniform ellipticity of $M$ implies its invertibility) and $\V'$ be  a finite dimensional linear subspace of $(L^2(\Omega))^d$. Define
\begin{equation}
\V:=\{M\zeta'\,:\, \zeta'\in \V'\}
\end{equation}
Assume furthermore that for all $w\in H^1_0\cap H^2(\Omega)$,
\begin{equation}\label{kjhekjwhe}
\inf_{\zeta' \in \V'} \|\nabla w-\zeta'\|_{(L^2(\Omega))}\leq C h \|\Delta w\|_{L^2(\Omega)}
\end{equation}
where $h$ is a small parameter (the resolution of the tessellation associated to $\V'$ for instance).
We remark here that $ \V'$ can  be viewed as the coarse scale $h$ approximation space (see example below).  The fine scale information from coefficients $a(x)$ is contained in the elements of the matrix $M$.  This is illustrated in the example below where $M=\nabla F$ for harmonic coordinates $F$. Therefore, the matrix  $M$ is determined by  $d$  harmonic coordinates that are analogues of
$d$ cell problems in periodic homogenization, and we call space $\V$  the ``minimal pre-computation space" since it requires minimal (namely $d$)  pre-computation of fine scales.

Then, we have the following theorem:

\begin{Theorem}\label{jkgkjghjg}  {\bf Approximation by ``minimal pre-computation space" }
If
\begin{itemize}
\item $a\cdot M$ is divergence free (as defined in sub-section \ref{jhsdgkjgewh3}).
\item The symmetric part of $a \cdot M$ satisfies the Cordes condition \eref{sshgdswsdsded7641} or the symmetric part of $a \cdot M$ satisfies
one of the inequalities of theorem \ref{Prop6s}.
\item The non-conforming error satisfies $\mathcal{K}_\V\leq C h^\alpha$ for some constant $C>0$ and $\alpha \in (0,1]$,
\end{itemize}
then
\begin{equation}\label{jakhgsjahhghsd}
\begin{split}
\|\nabla u - \zeta_{\V}\|_{(L^2(\Omega))^d}\leq C \|f\|_{L^2(\Omega)} h^\alpha
\end{split}
\end{equation}
where $u$ is the solution of \eref{scalarproblem0} and $\zeta_\V$ the solution of \eref{klwkjhelkjhde3},
\end{Theorem}
\begin{Remark}
The error estimate is given in the $L^2$ norm of $\nabla u - \zeta_{\V}$ because we wish to give a strong error estimate and if $\zeta_{\V}$ is not a gradient then the $L^2$ norm of
$(a(\nabla u - \zeta_{\V}))_{pot}$ is not equivalent to the $L^2$ norm $\nabla u - \zeta_{\V}$.
\end{Remark}
\begin{Remark}
It is, in fact, sufficient that the symmetric part of $a \cdot M$ satisfies
one of the inequalities of theorem  \ref{kshkheie} instead of \ref{Prop6s} for the validity of Theorem \ref{jkgkjghjg}. For the sake of clarity, we have used inequalities of theorem \ref{Prop6s}.
\end{Remark}
\begin{proof}
The proof is a direct consequence of theorem \ref{kjwhkejhwejd}; we simply need to bound $\mathcal{D}_{\V}$.  Since $\diiv (a\V)=\diiv(a\cdot M\V')$, it follows from equation
\ref{lkjedswddeewhhh3} that

\begin{equation}\label{lkjedswdadeewhccdhh3}
\mathcal{D}_{\V} \leq \sup_{w'\in H^2(\Omega)\cap H^1_0(\Omega)} \inf_{\zeta' \in \V'} \frac{\big\|(a\cdot M(\nabla u'-\zeta'))_{pot}\big\|_{(L^2(\Omega))^d}}{\|\Delta w'\|_{L^2(\Omega)}}
\end{equation}
where $u'$ in \eref{lkjedswddeewhhh3} is defined as the (weak) solution of $\diiv(a\cdot M\nabla u')=\Delta w'$ with Dirichlet boundary condition on $\partial \Omega$. Now, if symmetric part of $a \cdot M$ satisfies the Cordes condition \eref{sshgdswsdsded7641} or the symmetric part of $a \cdot M$ satisfies one of the inequalities of theorem \ref{Prop6s}, then $\|u'\|_{H^2}\leq C \|\Delta w'\|_{L^2}$ and we conclude using the approximation property \eref{kjhekjwhe}.

\end{proof}

 An example of $\V$ can be found in the discontinuous Galerkin method introduced in subsection 1.3 of \cite{MR2292954}. This method is also a generalization of the method II of \cite{MR1286212} to non-laminar media. In that method, we pre-compute
 the harmonic coordinates associated
with \eref{scalarproblem0}--i.e., the $d$-dimensional vector $F(x):=\big(F_1(x),\ldots,F_d(x)\big)$ where $F_i$ is a solution of
\begin{equation}\label{ksdkjskdh}
\begin{cases}
\diiv a \nabla F_i=0 \quad \text{in}\quad \Omega\\
F_i(x)=x_i \quad \text{on}\quad \partial \Omega.
\end{cases}
\end{equation}
Introducing $\Omega_h$, a regular tessellation of $\Omega$ of resolution $h$,
the elements of $\V$ are defined as $\nabla F (\nabla_c F)^{-1}\nabla \varphi$, where $\varphi$ is a piecewise linear function on $\Omega_h$
with Dirichlet boundary condition on $\Omega_h$ and $\nabla_c F$ is the gradient of the linear interpolation of $F$ over $\Omega_h$. In that example $a\cdot \nabla F$ is divergence-free and $\nabla F$ plays the role of $M$. The non-conforming error is controlled by the aspect ratios of the images of the triangles of $\Omega_h$ by $F$.
In  \cite{MR2292954}, the estimate \eref{jakhgsjahhghsd} is obtained using $F$ as a global change of coordinates that has no clear equivalent for tensorial equations, whereas the proof based on the flux-norm can be extended to tensorial equations.

\begin{figure}[h!]
\begin{center}
$
\begin{array}{cc}
  \psfig{figure=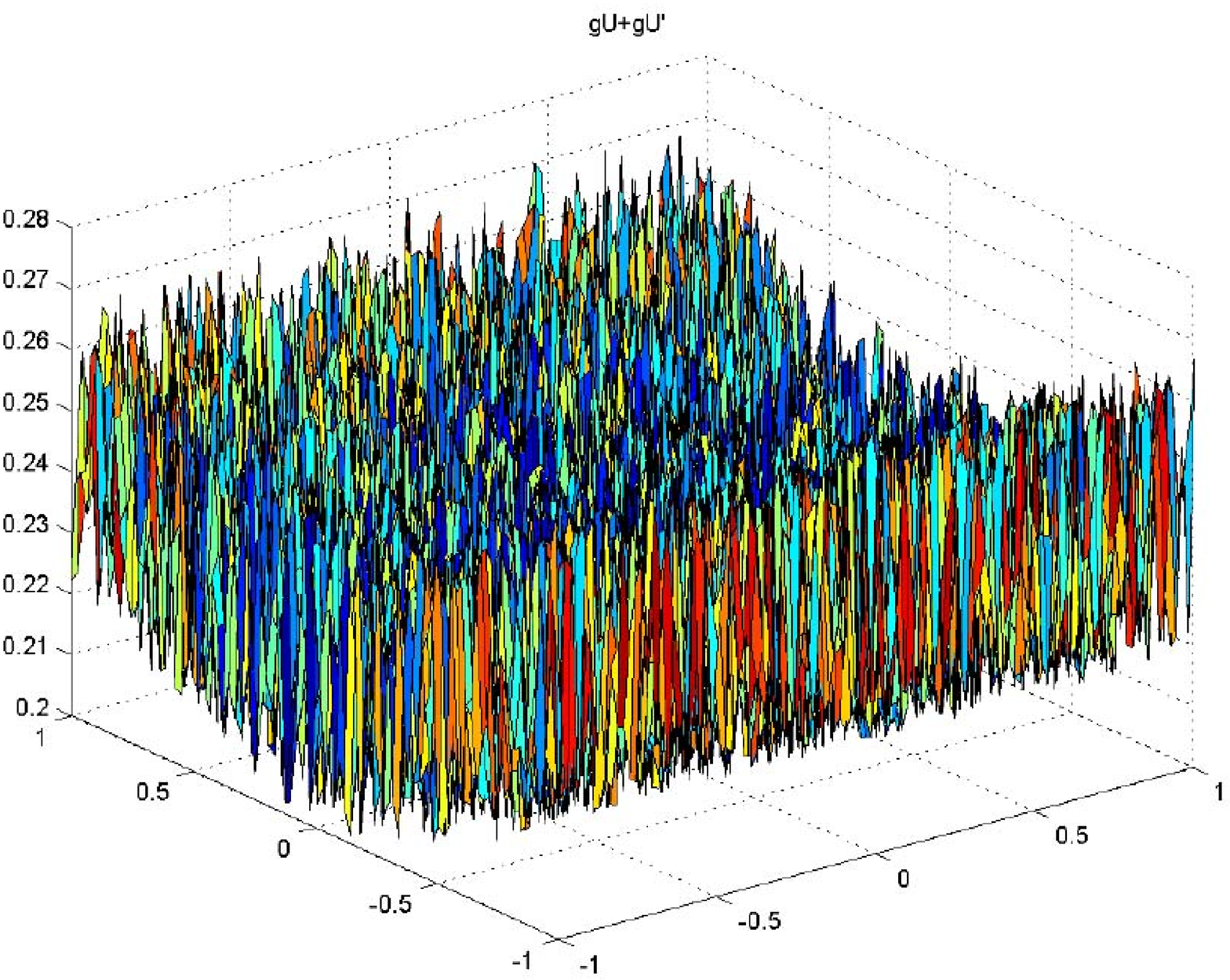,height=2.2in,silent=} & \psfig{figure=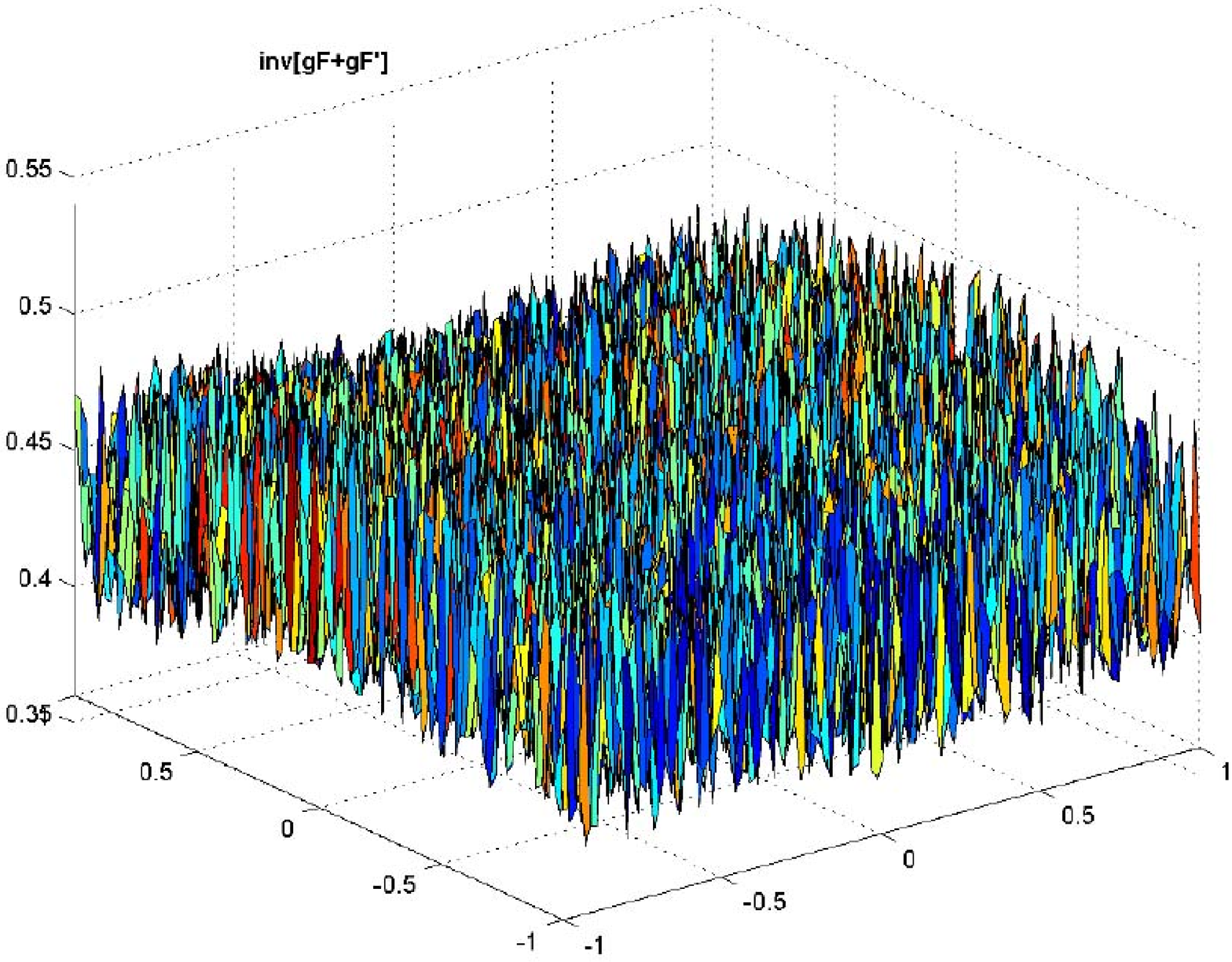,height=2.2in,silent=} \\
  \text{(a)} &
  \text{(b)} \\
  \psfig{figure=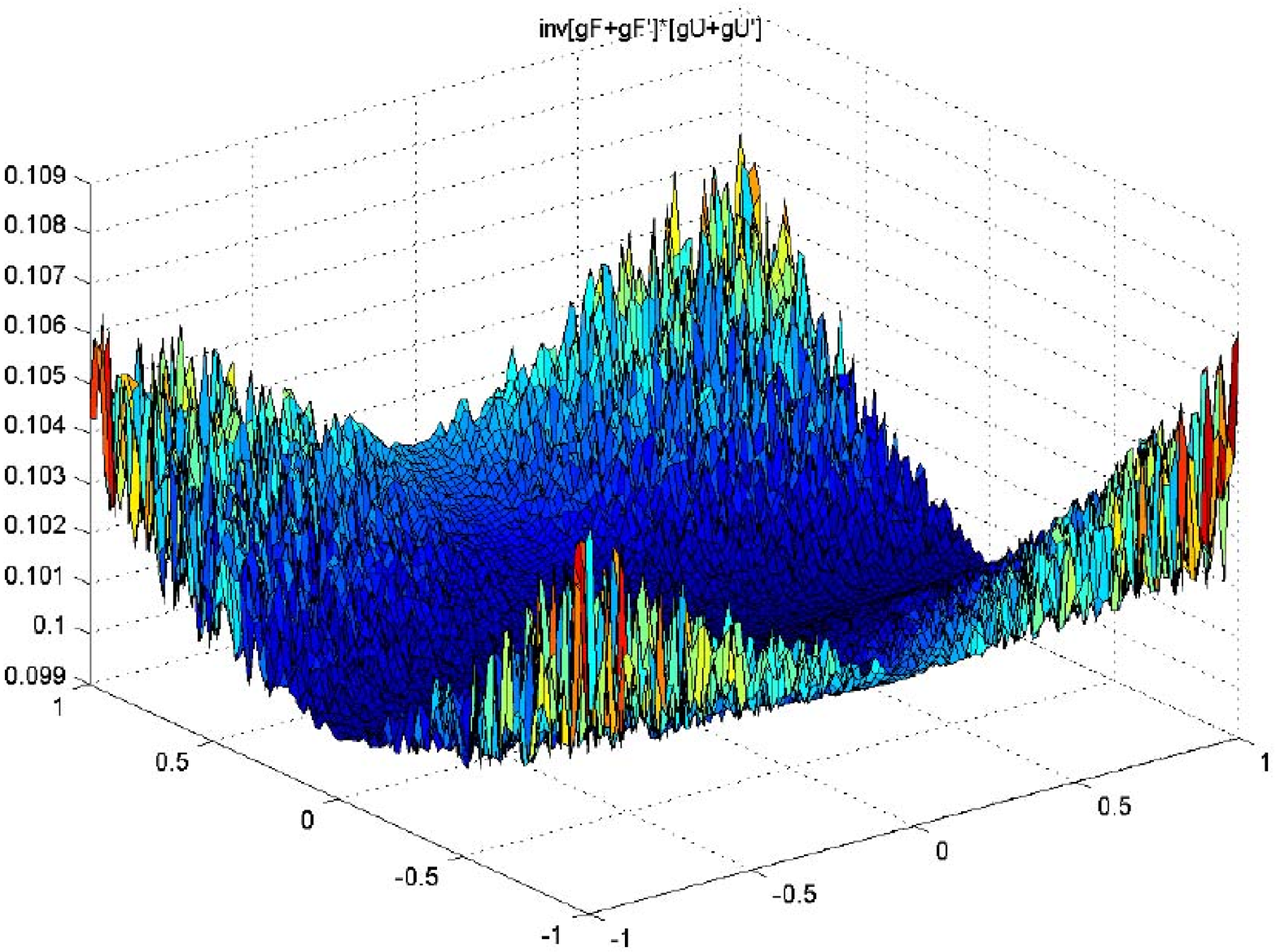,height=2.2in,silent=} & \\
  \text{(c)} &
\end{array}$
\end{center}
\caption{ \emph{Computation by Lei Zhang. The elasticity stiffness is obtained by choosing its coefficients to be random and oscillating over many overlapping scales.  Figure (a) and (b)  show
wild oscillations of one of the components of the strain tensor
$\nabla u + \nabla u^T$ ($u$ solves \eqref{elasticity}) and one of
the components of $(\nabla F + \nabla F^T)^{-1}$ ($F=\{F^{ij}\}$ is
defined by \eqref{elasticityh}).
Figure (c) illustrates one of the components of the product $(\nabla
F + \nabla F^T)^{-1}(\nabla u + \nabla u^T)$, which is smooth
if compared to (a) and (b). There is no smoothing near the boundary
due to sharp corners.}
 }\label{figupscaling}
\end{figure}

\subsection{Tensorial equations.}
The generalization of the results of this section to elasticity equations doesn't pose any difficulty. This generalization is simply based on theorem
\ref{sdjhskjddshskdhsdsdsdsxkhje} and the new class of inequalities introduced in subsection \ref{NewIneq2}. An example of numerical scheme can be found in \cite{OwDes09} for (non-linear) elasto-dynamics with rough elasticity coefficients. With elasticity equations harmonic coordinates are replaced by harmonic displacements, i.e. solutions of
\begin{equation}\label{elasticityh}
    \begin{cases}
    -\diiv (C(x) \nabla F^{kl})=0 \quad &  x \in \Omega\\
    F^{kl}=\frac{x_k e_l+x_l e_k}{2} \quad & \text{on } \partial \Omega.
    \end{cases}
\end{equation}
and strains $\varepsilon(u)$ are approximated by a finite dimension linear space $\V$ with elements of the form
$\varepsilon(F):(\varepsilon_c F)^{-1}(\varepsilon(\varphi))$ where the $\varphi$ are piecewise linear displacements on $\Omega_h$,
$\varepsilon_c F$ is the strain of the linear interpolation of $F$ over $\Omega_h$ and
$\varepsilon(F)$ denotes the $d\times d\times d\times d$ tensor with entries
\begin{equation}
\varepsilon(F)_{i,j,k,l}:=\frac{\partial_i F_j^{kl}+\partial_j
F_i^{kl}}{2}.
\end{equation}
Here, $C:\varepsilon(F)$ is divergence-free and plays the role of $M$; furthermore, the regularization property observed in the scalar case  \cite{MR2292954} is also observed in the tensorial case by taking the product $(\varepsilon(F))^{-1}\varepsilon(u)$ (figure \ref{figupscaling}).

\section{Relations with homogenization theory and other works.}\label{kjkgjkjgkjgjhg}

We first show  how our approach is related to  homogenization theory. To this end,  we
\begin{itemize}
\item Describe the notion of a thin subspace,  which is pivotal in our work and show that this notion was implicitly present in classical periodic homogenization.

\item Show the analogy between the basis functions in Theorems  \ref{jhdjkhsgdkjgeyhde} and \ref{Corol1}, harmonic coordinates \eqref{harmonic}, and solutions of the cell problems in periodic and random homogenization.

\item Explain relations between our work and general  abstract operator homogenization approaches.

\end{itemize}

 \paragraph{ The thin subspace notion.}  A key ingredient of the  proofs of the main approximation  Theorems  \ref{jhdjkhsgdkjgeyhde} and \ref{Corol1} is  the transfer property introduced in theorems \ref{sdjhskjdhskdhsdsdsdsxkhje} and \ref{sdjhskzzhdsdkhje}.  Roughly speaking, it shows how  a standard (``easy'') error estimate   in the  case of  smooth coefficients  provides an error estimate  in the  case of arbitrarily rough (``bad'')  coefficients due to an appropriate choice of the finite-dimensional approximation space.

 The transfer property in turn is based on the notion of a ``thin''   subspace whose essence can be explained as follows.  Let us consider the scalar divergence form elliptic problem \eref{scalarproblem0}. First, observe that as $f$ spans $H^{-1}(\Omega)$, $u$ spans $H^1_0(\Omega)$, i.e.
the operator $L^{-1}:=(-\diiv a \nabla)^{-1}$ defines a bijection from $H^{-1}(\Omega)$ onto $H^1_0(\Omega)$. Next, observe that as $f$ spans $L^2(\Omega)$, $u$ spans a subspace $V$ of $H^1_0(\Omega)$, i.e. $L^{-1}:=(-\diiv a \nabla)^{-1}$ defines a bijection from $L^2(\Omega)$ onto $V$.
How ``thin'' is that space compared to
 $H^1_0(\Omega)$? If $a=I_d$, then  $V=V':=H^1_0(\Omega)\cap H^2(\Omega)$,
i.e. a much ``thinner'' space than $H^1_0(\Omega)$, namely $V$ is ``as thin as $H^2$".

The proofs of Theorems \ref{jhdjkhsgdkjgeyhde} and \ref{Corol1} also use the transfer property  for
 finite-dimensional approximation spaces that depend of  a small parameter $h$. Namely, there exists a finite $O(|\Omega|/h^d)$ dimensional subspace $V_h'$ of $H^1_0(\Omega)$ such that all elements of $V'$ are in $H^1$-norm distance at most $h$ from $V_h'$  (an example of spaces $V_h'$ and $V'$ are  the spaces $\mathcal{L}_0^h$ (used in subsection \ref{khdkhdlkwhe}) and $H^2(\Omega)$ respectively).

Section \ref{jkgjkhgjgu} shows
 that when the entries of $a$ are only assumed to be bounded,  the solution space $V$ is
isomorphic to $V'=H^1_0(\Omega)\cap
 H^2(\Omega)$, that is  for  arbitrarily rough coefficients the  approximation space  $V$ is still ``as thin as'' $H^2$ (isomorphic to $H^2$). Moreover, the transfer property introduced in theorem \ref{sdjhskjdhskdhsdsdsdsxkhje} allows us to explicitly construct a  finite-dimensional space $V_h$, isomorphic to $V_h'$, such that all elements of $V$ are in $H^1$ norm  distance at most $h$ from $V_h$.

We next show that the thin subspace notion  is implicitly present in classical homogenization problem when
$a$ is periodic, with period $\epsilon$, i.e. when equation \eref{scalarproblem0} is of the form
\begin{equation}\label{classicperiodic}
-\text{div} (a(x/ \ve) \nabla u^\ve(x))=f(x), \text { in } \Omega \subset \mathbb{R}^d
\end{equation}

 From the two-scale asymptotic expansion ansatz justified in periodic homogenization (e.g.,\cite{BeLiPa78}, \cite{JiKoOl91}, \cite{BaPa90}, \cite{CioDon99}, \cite{MR1859696}), we know that $u^\epsilon$ can be approximated  in $H^1$ norm by (modulo boundary correctors which we do not discuss here for the sake of simplicity of  presentation)
   \begin{equation}\label{corrector}
  \begin{array}{rl}
  \hat{u}^\epsilon(x) = \hat u(x) +\epsilon \sum_{k=1}^{d} \chi_k \left(\frac{x}{\epsilon}\right) \frac{\partial \hat u(x)}{\partial x_k}.
  \end{array}
  \end{equation}
where  $\hat u$ is the solution of the homogenized problem
\begin{equation}\label{classichomogenized}
-\text{div} \left( \hat{a} \nabla \hat{u} \right) = f(x)
\end{equation}
 with constant homogenized (effective) coefficient $\hat{a}$.  Here, the  exact solution $u^\ve$ has both fine, $O(\ve)$, and coarse, $O(1)$, variations (oscillations), while the homogenized solutions $\hat u(x)$ has only coarse scale variations.  In \eqref{classichomogenized}, the periodic functions $\chi_k $
  are solutions of the cell-problems
    \begin{equation}\label{cellpb1}
  \begin{array}{rl}
  \diiv (a(y) (e_k+\nabla \chi_k(y)))=0
  \end{array}
  \end{equation}
defined on the torus of dimension $d$.  The second term in the right hand side of \eqref{classichomogenized} is known as a corrector; it has both fine  and coarse scales, but the fine scales ($\ve$- oscillations) enter in a controlled way via  $d$ solutions of the cell problems that do not depend on $f(x)$ and the domain $\Omega$ (so that $\chi_k $ are  completely determined by the microstructure $a(x)$).

Furthermore, the two-scale convergence approach \cite{Ngu90, Al92}  provides a simple and  elegant description of the approximation   for  the gradients $\nabla u_\epsilon$. Namely,
for every sufficiently smooth $\phi(x,y)$ which is also periodic in $\epsilon$,  $u_\epsilon \rightarrow \hat{u} $ weakly in $H^1(\Omega)$ and
\begin{equation}\label{two-scale}
\int_{\Omega} \phi(x,\frac{x}{\epsilon}) \nabla u_\epsilon \rightarrow \int_{\Omega \times \T^d} \phi(x,y) (I_d+\nabla \chi_.(y))\nabla \hat{u}(x)
\end{equation}
where $\T^d$ is the torus of dimension $d$.  Thus $\nabla u_\epsilon$ can be approximated in the sense of \eqref{two-scale} by functions of the form $(I_d+\nabla \chi_.(x/\ve))\nabla \hat{u}(x)$

The latter observation combined with  \eqref{corrector}  shows that classical homogenization results can be viewed as follows.
The solution space $V$ for the problem   \eref{classicperiodic}
can be approximated in $H^1$ norm by a (``thin'') subspace of $H^1(\Omega)$ parameterized by solutions of \eqref{classichomogenized} which are in $H^2 \cap H^1_0(\Omega)$. Moreover, homogenization theory shows us how to construct an approximation of the space $V$. Indeed, \eqref{corrector}, \eqref{two-scale} show that this approximation space  is determined by $\hat u(x)$ and solutions of the cell problems  \eref{cellpb1} over one period.

\paragraph{Cell problems.}
This periodic homogenization scheme was generalized to stationary ergodic coefficients $a(x/ \ve, \omega)$, with $\omega$ in some probability space. Here, there are also analogs of the cell problems  that require the solution of $d$ different boundary value problems for the PDE $\text{div} (a(x, \omega)\nabla u_i)=0, i=1, \dots,d$ in a cube of size $R \to \infty$ for a typical realization $\omega$ \cite{JikKozOle94,PapVar82}.  The solutions  $u_i$  must be pre-computed in order to obtain a homogenized PDE and an approximate solution  just like in the periodic problem, which is why they still are called  cell problems even though there is no  actual periodicity cell in the microstructure. Here, the coefficients in the cell problems   have both fine and coarse scales, as in the original PDE $\text{div} (a(x /\ve, \omega)\nabla u^\ve)=f$. The advantage of solving  cell problems  numerically (in both  the periodic and random case) comes  when, e.g.,  we need to solve for many  different $f$ or for a corresponding evolution problem $\partial_t u^\ve=\text{div} (a(x /\ve, \omega)\nabla u^\ve) -f$ when updating in fine time scales.

Note that the major difficulty in advancing from periodic to random homogenization was to understand {\it what is the proper analog of the periodic cell problem}. In this work, we ask  a similar question-- what are the analogs of cell problems for most general arbitrarily rough coefficients?  For the problem \eqref{scalarproblem0}, we provide two answers.

 First, in Theorem \ref{Corol1}, we  introduce functions
$\theta_k(x)$  that are analogs  of the cell problems since they  are determined by the  coefficients $a(x)$ but  do not depend on $f(x)$. Note, that these ``generalized cell problems'' must depend on the domain $\Omega$ since the coefficients are no longer translationally invariant (as in periodic and random stationary case).  This dependence  enters via $\Psi_k$ in \eqref{kjskske23} using the transfer property. The approximation space in   Theorem \ref{Corol1}  is $\Theta_h$ defined in \eref{harmonsedcalwsar}-\eqref{harmseeddcalwsar}. Similarly, in Theorem \ref{jhdjkhsgdkjgeyhde}, we introduce  functions $\Phi_k(x)$  that solve  \eqref{harmonsedcalwsarm}   with localized right hand side  that are also analogs of the cell problems.

Secondly, the {\it harmonic coordinates}  \eqref{harmonic} provide yet  another analog of   cell problems in classical homogenization. Their advantage is obvious-- there is only $d$ of them, whereas  in Theorems  \ref{Corol1}  and \ref{jhdjkhsgdkjgeyhde}, the number of cell problems (number of elements in the basis of $V_h$) is $N(h)$. In fact, since in the simplest case of periodic homogenization $d$ cell problems \eqref{cellpb1} must be used, one should expect that for the more general coefficients $a(x)$,  $d$ would be the minimal number of cell problems.
On the other hand,  the  finite dimensional approximation  based on harmonic coordinates, in general, is not   direct (involves the non-conforming error) as one   can be seen from Theorem  \ref{jkgkjghjg} and example right after this Theorem.

 Harmonic coordinates play an important role in various homogenization approaches, both theoretical and numerical, which is why we present here a short account of their development. Recall that harmonic coordinates  were  introduced in \cite{MR542557} in the context of random homogenization. Next, harmonic coordinates have been used in one dimensional and quasi-one dimensional divergence form elliptic problems  \cite{ BabOsb83, MR1286212}, allowing for efficient finite
dimensional approximations.

The idea of using particular solutions  in numerical homogenization to approximate the solution space of \eqref{scalarproblem0}  have been first proposed in reservoir modeling
in the 1980s \cite{BraWu09}, \cite{WhHo87} (in which a global scale-up method was introduced based on generic flow solutions i.e.,
flows calculated from generic boundary conditions). Its rigorous mathematical analysis was done
only recently \cite{MR2292954}. In  \cite{MR2292954},  it was shown   that if $a(x)$ is not periodic but satisfies the  Cordes conditions (a restriction on anisotropy for $d \geq 3$, no restriction for $d=2$ and convex domains $\Omega$), then    the (``thin'') approximation space $V$ can be
 constructed from any set of $d$ ``linearly independent'' solutions of
 \eref{scalarproblem0} (harmonic coordinates $F$, for instance by observing that $u\circ F^{-1}$ spans $H^2 \cap H^1_0(\Omega)$ as $f$ spans $L^2(\Omega)$).  In the present work (section \ref{lkjlsdhkj}) for
 elasticity problems, we show that $d(d+1)/2$ ``linearly independent'' solutions are
 required (equation \eref{elasticityh}).

In \cite{AllBri05},  the solution space $V$ is approximated
 by composing splines with local harmonic coordinates (leading to higher accuracy), and a proof of convergence is given for
periodic media. Harmonic coordinates have been motivated and linked to periodic homogenization in \cite{AllBri05} by observing that equation \eref{corrector} can in fact be seen as a Taylor expansion of $\hat u(x+\chi.(x)) $ where $x+\chi_.(x)$ is harmonic, i.e., satisfies \eqref{cellpb1}. It is also observed in \cite{AllBri05} that replacing
$x+\chi_.(x)$ by global harmonic coordinates $F(x)$ automatically enforces Dirichlet boundary conditions on $\hat{u}^\epsilon$.

More recently, in \cite{MR2322432, MR2281625, BraWu09},
the idea of a global change of coordinates analogous to harmonic coordinates was implemented numerically in order to
up-scale porous media flows. We refer, in particular, to a recent review article \cite{BraWu09}   for an overview of some main challenges in reservoir modeling and a description of global scale-up strategies based on generic flows.

\paragraph{Abstract operator homogenization approaches and their relation to our work.} Recall that the theory of homogenization in its most general formulation is
based on abstract operator convergence, --i.e.,  $G$-convergence for
symmetric operators, $H$-convergence for non-symmetric operators and
$\Gamma$-convergence for variational problems. We refer to the work
of De Giorgi, Spagnolo, Murat, Tartar, Pankov and many others
\cite{Mur78, Gio75, MR630747, MR0477444, MR0240443, MR506997,
MR1968440}). $H$, $G$ and $\Gamma$-convergence allows one to obtain the
convergence of a sequence of operators parameterized by $\epsilon$ under
very weak assumptions on the coefficients.  The concepts of ``thin'' space and generalized cell problems  are implicitly present  in this most general form of homogenization theory through
the introduction of oscillating test
functions in $H$-convergence
\cite{Mur78} (see also related work on G-convergence \cite{MR0240443,Gio75}).
Furthermore, the so called multiscale finite element method \cite{MR1455261, MR1898136} can be seen as a numerical generalization of this idea of oscillating test functions
with the purpose of constructing a numerical (finite dimensional) approximation of the ``thin'' space of solutions $V$. We refer to \cite{AnGlo06} for convergence results on the multiscale finite element method in the framework of $G$ and $\Gamma$-convergence.

Observe that in most engineering problems, one has to deal with a
given medium and not with a family of media.  In particular, in those problems, it is not
possible to find a small parameter $\epsilon$ intrinsic to the
medium  with respect to which one could perform an asymptotic
analysis. Indeed, given a medium that is not periodic or stationary ergodic, it is not clear how to define a family of operators $A_\epsilon$. Moreover, the definition of oscillating test functions involves the limiting (homogenized) operator $\hat A$. While this  works well for the proof of the  abstract convergence results, in practice only the coefficients $A$ are known (computing $\hat A$ may not be possible), and our approach allows one to construct the approximate (upscaled) solution  from the given coefficients without
constructing $\hat A$.
Hence, the main difference between $H$ or $G$ convergence and present work is that instead of characterizing the limit of an $\epsilon$-family of  boundary value problems we are approximating the solution to  a given problem with a  finite-dimensional operator with explicit error estimates (i.e. constructing an explicit finite dimensional approximation of the ``thin'' solution space $V$).

We refer to \cite{DeDoOw09} for
for an explicit construction of $\hat{A}$ with rough coefficients $a$ in two dimensions. In particular, it is shown in \cite{DeDoOw09} that
conductivity coefficients $a$ are  in one-to-one correspondence with convex functions $s(x)$ over the domain $\Omega$ and that homogenization of $a$ is equivalent to the linear interpolation over triangulations of $\Omega$ re-expressed using convex functions.

The thin subspace idea introduced in this section can be used to develop coarse graining numerical schemes through an energy matching principle.
We refer to \cite{OwDes09} for elasticity equations and to \cite{ZhBFOww09}  for atomistic to continuum models (with non-crystalline structures).

\paragraph{Other related works.}
By now, the field of asymptotic and numerical homogenization with non periodic
coefficients has become large enough that it is not possible to
cite all contributors. Therefore, we will restrict our attention to
works directly related to our work.

- In  \cite{MR2314852, EnSou08}, the structure of the
medium is numerically decomposed into a micro-scale and a
macro-scale (meso-scale) and solutions of cell problems are computed
on the micro-scale, providing local homogenized matrices that are
transferred (up-scaled) to the macro-scale grid. This procedure
allows one to obtain rigorous homogenization results with controlled
error estimates for non periodic media of the form
$a(x,\frac{x}{\epsilon})$ (where $a(x,y)$ is assumed to be smooth in
$x$ and periodic or ergodic with specific mixing properties in $y$).
Moreover, it is shown that the numerical algorithms associated with HMM and MsFEM can be
implemented for a class of coefficients that is much broader than $a(x,\frac{x}{\epsilon})$. We refer to \cite{AnGlo06} for convergence results on the Heterogeneous Multiscale Method in the framework of $G$ and $\Gamma$-convergence.

- More recent work includes an adaptive projection based method
\cite{MR2399542}, which is consistent with homogenization when there is scale
separation, leading to adaptive algorithms for solving problems with
no clear scale separation;  fast and sparse chaos approximations of
elliptic problems with stochastic coefficients \cite{MR2317004,
MR2399150}; finite difference approximations of fully nonlinear,
uniformly elliptic PDEs with Lipschitz continuous viscosity
solutions \cite{MR2361302} and operator splitting methods
\cite{MR2342991, MR2231859}.

- We refer the  reader to \cite{MR2283892} and the references therein for a series of computational papers on  cost versus accuracy capabilities for the generalized FEM.

- We  refer to \cite{MR2334772, MR2284699} (and references therein) for most recent results on homogenization of  scalar divergence-form elliptic operators with stochastic coefficients. Here the stochastic coefficients $a(x /\ve, \omega)$ are obtained from stochastic deformations (using random diffeomorphisms) of the periodic and stationary ergodic setting.

 - We refer the  reader to \cite{ChuHou09}, \cite{EffGaWu09} and \cite{MR1771781} for recent results on adaptive finite element methods for high contrast media.
 Observe that in \cite{ChuHou09}, contrast independent error estimates are obtained for a domain with high contrast inclusions by dividing the energy norm by the minimal value of $a$ over the domain $\Omega$.  The strategy  of \cite{MR1771781}  is to first prove a priori and a posteriori estimates
that are contrast independent and then construct a finite element mesh  adaptively such that the error is the smallest possible for a fixed number of degrees of freedom. In \cite{MR1771781}, the energy norm of the error is bounded by terms that are appropriately and explicitly weighted by $a$ to obtain error constants independent of the variability of $a$.

\section{Conclusions}

In this paper, we have been primarily concerned with  the following theoretical results:

\begin{itemize}
\item  the flux norm introduced in section \ref{sec1} and  its transfer property (theorems \ref{sdjhskjdhskdhsdsdsdsxkhje} and \ref{sdjhskzzhdsdkhje})

\item  the resulting theoretical Galerkin method with contrast independent accuracy  for linear PDEs with arbitrarily rough coefficients  (theorems \ref{jhdjkhsgdkjgeyhde}, \ref{Corol1}, \ref{Corol1del}, \ref{sids5sasawqebyuyuyyassdud})

\item a new class of
inequalities (section \ref{inequalities}) and the introduction of a Cordes-type condition for non-divergence tensorial elliptic equations (theorem \ref{kahlkguelwhu}).

\end{itemize}
The development of numerical techniques based on  these results is in progress  and  will be addressed elsewhere.

\section{Appendix}

\subsection{Extension to non-zero boundary conditions.}\label{NonZeroBC}
The analysis performed in  section \ref{jkgjkhgjgu} and the previous one can naturally be extended to other types of boundary conditions (Neumann or Dirichlet). To support our claim, we will provide here this extension in the scalar case with non-zero Neumann boundary conditions. By linearity, solutions of \eref{scalarproblem0} with nonzero boundary conditions can be written as the sum of the solution with zero-boundary condition and $  -\diiv \Big(a(x)  \nabla u(x)\Big)=f(x)$ and the solution with non-zero boundary condition and $  -\diiv \Big(a(x)  \nabla u(x)\Big)=0$. Hence, we will restrict our analysis to solutions of
\begin{equation}\label{scalarproblem0h}
\begin{cases}
    -\diiv \Big(a(x)  \nabla u(x)\Big)=0 \quad  x \in \Omega;  \;   a(x)=\{a_{ij} \in L^{\infty}(\Omega)\}\\
    n. (a \nabla u)(x)=h(x)  \quad  x \in \partial \Omega;   \quad h\in L^2(\partial \Omega),
    \end{cases}
\end{equation}
We assume $d\geq 2$.
Write $L^2_{potn}(\Omega)$ the closure of $\{\nabla v\,:\, v\in H^1(\Omega)\}$ in $(L^2(\Omega))^d$. Note that the difference with $L^2_{pot}(\Omega)$ lies in replacement of $v\in H^1_0(\Omega)$ by $v\in H^1(\Omega)$ in the definition of $L^2_{potn}(\Omega)$.
 For $\xi$ in $(L^2(\Omega))^d$, write $\xi_{potn}$ its orthogonal projection on
$L^2_{potn}(\Omega)$. For $v\in H^1(\Omega)$, write
\begin{equation}
\|v\|_{a\f,\text{n}}^2:=\int_{\Omega} (a\nabla v)_{potn}^2
\end{equation}
It follows from the following lemma that the $\|v\|_{a\f,\text{n}}$-norm is equivalent to the $H^1$-norm.
\begin{Lemma}
For $v\in H^1(\Omega)$
\begin{equation}\label{kjhgdksjghd1}
\lambda_{\min}(a) \|\nabla v\|_{(L^2(\Omega))^d}\leq \|v\|_{a\f,\text{n}} \leq \lambda_{\max}(a) \|\nabla v\|_{(L^2(\Omega))^d}
\end{equation}
\end{Lemma}
\begin{proof}
The proof of the right hand side of \eref{kjhgdksjghd1} is straightforward. The proof of the left hand side follows from
\begin{equation}\label{kjhgdksjgsshd0d}
\int_{\Omega} (\nabla v)^T a \nabla v\leq \|\nabla v\|_{(L^2(\Omega))^d}\|v\|_{a\f,\text{n}}
\end{equation}
\end{proof}

Write $\Lambda$ the Dirichlet  to Neumann map  mapping $v|_{\partial \Omega}$ onto $n\cdot \nabla v$ on $\partial \Omega$ where $v$ is the solution of $\Delta v=0$ in $\Omega$ (and $n$, the exterior normal to the boundary $\partial \Omega$). Write $\Psi_k$ the orthonormal eigenvectors of $\Lambda$ and $\lambda_k$ the associated increasing and positive eigenvalues.
Let $V_n$ be the finite dimensional subspace of $H^1(\Omega)$ formed by the linear span of $v_1,\ldots,v_n$, where $v_k$ is the solution of
\eref{scalarproblem0h} with $h=\Psi_k$.

\begin{Theorem}\label{thm1}
 For $h\in L^2(\Omega)$, let $u$ be the
solution of \eref{scalarproblem0h}. Then,
\begin{equation}\label{jksdhsjdhdf}
\sup_{h\in L^2(\partial \Omega)} \inf_{v \in V_n} \frac{\|u-v\|_{a\f,\text{n}}}{\|h\|_{L^2(\partial \Omega)}}=\frac{1}{(4\pi)^\frac{1}{4}}\Big(\frac{|\partial \Omega|}{\Gamma(1+\frac{d-1}{2})n }\Big)^\frac{1}{2(d-1)}
\end{equation}
Furthermore, the space $V_n$ leads to the smallest possible constant in the right hand side of \eref{jksdhsjdhdf} among all subspaces of $H^1(\Omega)$ with $n$ elements.
\end{Theorem}
\begin{Remark}
We also refer to \cite{MR2430615} for the related introduction of the penetration function, which measures the effect of the boundary data on the energy of  solutions of \eref{scalarproblem0h} and is used to assess the accuracy of global-local approaches for recovering local solution features from coarse grained solutions such as those delivered by homogenization theory.
\end{Remark}
\begin{proof}
For $h\in L^2(\partial \Omega)$, let  $w^h$ be the solution of
\begin{equation}\label{scalarproblem0hh}
\begin{cases}
    \Delta w^h(x)=0 \quad  x \in \Omega;  \\
    n\cdot \nabla w^h (x)=h(x)  \quad  x \in \partial \Omega;   \quad h\in L^2(\partial \Omega),
    \end{cases}
\end{equation}
Let $W_n$ be a finite dimension subspace of $H^1(\Omega)$ of elements $v$ satisfying $\diiv(a\nabla v)=0$ in $\Omega$.
To prove Theorem \ref{thm1}, we will first prove lemmas \ref{h1} and \ref{h2} given below.

\begin{Lemma}\label{h1}
 For $h\in L^2(\partial \Omega)$, let  $u$ be the solution of \eref{scalarproblem0h} and $w^h$ be the solution of \eref{scalarproblem0hh}. We have,
\begin{equation}
\sup_{h\in L^2(\partial \Omega)} \inf_{v \in W_n} \frac{\|u-v\|_{a\f,\text{n}}}{\|h\|_{L^2(\partial \Omega)}}=\sup_{h \in L^2(\Omega)} \inf_{v \in W_n} \frac{\|(\nabla w^h-a \nabla v)_{potn}\|_{(L^2(\Omega))^d}}{\|h\|_{L^2(\partial \Omega)}}
\end{equation}
\end{Lemma}
\begin{proof}
The proof follows by observing that for all $\varphi \in H^1(\Omega)$
\begin{equation}
\int_{\Omega}\nabla \varphi (\nabla w^h-a \nabla u)=\int_{\partial \Omega} \varphi\, n\cdot(\nabla w^h-a \nabla u)=0
\end{equation}
\end{proof}
For $v\in H^1(\Omega)$, write $h_v:=n.a\nabla v$ defined on $\partial \Omega$
\begin{Lemma}\label{h2}
For $h\in L^2(\partial \Omega)$ and $v\in H^1(\Omega)$,
\begin{equation}
\|(\nabla w^h-a \nabla v)_{potn}\|_{(L^2(\Omega))^d}^2=\int_{\partial \Omega} (h-h^v) \Lambda^{-1}(h-h^v)
\end{equation}
\end{Lemma}
\begin{proof}
The proof is obtained by observing that
\begin{equation}
\begin{split}
\|(\nabla w^h-a \nabla v)_{potn}\|_{(L^2(\Omega))^d}&=\sup_{\varphi \in H^1(\Omega)} \frac{\int_\Omega \nabla \varphi (\nabla w^h-a \nabla v)}{\|\nabla \varphi\|_{(L^2(\Omega))^d}}\\
&=\sup_{\varphi \in H^1(\Omega)} \frac{\int_{\partial \Omega}  \varphi (h-h_v)}{\|\nabla \varphi\|_{(L^2(\Omega))^d}}\\
&=\sup_{\varphi \in H^1(\Omega)} \frac{\int_{\Omega}  \nabla \varphi \nabla w^{(h-h_v)}}{\|\nabla \varphi\|_{(L^2(\Omega))^d}}\\&=\|\nabla w^{h-h_v}\|_{(L^2(\Omega))^d}\\
&=\Big(\int_{\partial \Omega}  (h-h^v) \Lambda^{-1}(h-h^v)\Big)^\frac{1}{2}
\end{split}
\end{equation}
\end{proof}

Write $Z_n$ the subspace of $H^{-\frac{1}{2}}(\partial \Omega)$ induced by elements $n\cdot a\nabla v$ defined on $\partial \Omega$ for $v\in W_n$. Write $h_k$ the coefficients of $h$ in the basis $\Psi_k$, i.e. $h_k:=\int_{\partial \Omega} h \Psi_k$,
It follows from lemmas \ref{h1} and \ref{h2} that

\begin{equation}
\sup_{h\in L^2(\partial \Omega)} \inf_{v \in W_n} \frac{\|u-v\|_{a\f,\text{n}}}{\|h\|_{L^2(\partial\Omega)}}=\sup_{h \in L^2(\partial \Omega)} \inf_{v \in W_n} \Big(\frac{\sum_{k=1}^\infty \frac{1}{\lambda_k} (h_k-h^v_k)^2}{\sum_{k=1}^\infty h_k^2}\Big)^\frac{1}{2}
\end{equation}
It follows that an  space $W_n$ with optimal approximation constant is obtained when  $Z_n$ is the linear span of $(\Psi_1,\ldots,\Psi_n)$. In that case

\begin{equation}
\sup_{h\in L^2(\partial \Omega)} \inf_{v \in W_n} \frac{\|u-v\|_{a\f,\text{n}}}{\|h\|_{L^2(\partial \Omega)}}=\frac{1}{(\lambda_{n+1})^\frac{1}{2}}
\end{equation}

We deduce equation \eref{jksdhsjdhdf}  using the following Weyl's asymptotic for $\Lambda$ (as a first order pseudo-differential operator, see for instance \cite{Kotiuga06}, \cite{MR1029119} and \cite{MR1862026})

\begin{equation}\label{sssddesqwsdedudpp}
\begin{split}
(\lambda_k)^2 \sim  4\pi
\Big(\frac{\Gamma(1+\frac{d-1}{2})k}{|\partial \Omega|}\Big)^\frac{2}{d-1}.
\end{split}
\end{equation}

This concludes the proof of Theorem \ref{thm1}.

\end{proof}

\subsection{Proof of lemma \ref{kahldssddkguelwhu}}\label{ell1}

Let $u$ be the solution of $\mathcal{L}u=f$ with Dirichlet boundary
condition (assume that it exists). Since $\alpha>0$, the solvability of \eref{kljdwedddlkwjewlkjer}
is equivalent to finding $u\in H^2\cap H^1_0(\Omega)$ such that
\begin{equation}
\Delta u=\alpha f+\Delta u-\alpha \mathcal{L}u
\end{equation}

Consider the mapping $T:H^2\cap H^1_0(\Omega)\rightarrow
H^2\cap H^1_0(\Omega)$ defined by $v=Tw$ where $v$ be the unique
solution of the Dirichlet problem for Poisson equation
\begin{equation}
\Delta v=\alpha f+\Delta w-\alpha \mathcal{L}w
\end{equation}

Let us now show that for $\beta_a<1$, $T$ is a contraction.

\begin{equation}
\begin{split}
\big\|T w_1-T w_2\big\|_{H^2\cap
H^1_0(\Omega)}=\|v_1-v_2\|_{H^2\cap H^1_0(\Omega)}
\end{split}
\end{equation}
Using the convexity of $\Omega$, one obtains the following classical
inequality satisfied by the Laplace operator (see lemma 1.2.2 of
\cite{MPG00})
\begin{equation}
\begin{split}
\|v_1-v_2\|_{H^2\cap H^1_0(\Omega)}\leq \|\Delta
(v_1-v_2)\|_{L^2(\Omega)}
\end{split}
\end{equation}
Hence,
\begin{equation}
\begin{split}
\big\|T w_1-T w_2\big\|_{H^2\cap H^1_0(\Omega)}^2\leq &\|\Delta
(w_1-w_2)-\alpha \mathcal{L}(w_1-w_2)\|_{L^2(\Omega)}^2\\
\leq &  \Big\|\sum_{i,j=1}^d
\big(\delta_{ij}-\alpha a_{ij}\big)
\partial_i
\partial_j (w_1-w_2)\Big\|_{L^2(\Omega)}^2
\end{split}
\end{equation}
Using Cauchy-Schwarz  inequality we obtain that
\begin{equation}
\begin{split}
\big\|T w_1-T w_2\big\|_{H^2\cap H^1_0(\Omega)}^2\leq &
\int_{\Omega}  \big(\sum_{i,j=1}^d
(\delta_{ij}-\alpha a_{ij})^2\big)
 \\& \big(\sum_{i,j=1}^d
(\partial_i
\partial_j (w_1^l-w_2^l))^2 \big)
\end{split}
\end{equation}
Hence observing that
\begin{equation}
\esssup_{\Omega}\big(\sum_{i,j=1}^d
(\delta_{ij}-\alpha a_{ij})^2\big)=\beta_a
\end{equation}
we obtain that
\begin{equation}
\begin{split}
\big\|T w_1-T w_2\big\|_{H^2\cap H^1_0(\Omega)}^2\leq
\esssup_{x\in \Omega}\beta_{a}(x) \big\|w_1-
w_2\big\|_{H^2\cap H^1_0(\Omega)}^2
\end{split}
\end{equation}
It follows that if $\beta_C<1$, then $T$ is a
contraction and we obtain the existence and solution of
\eref{kljdwedddlkwjewlkjer} through the fixed point theorem. Moreover,

\begin{equation}
\|\Delta u\|_{L^2(\Omega)}\leq \|\alpha
f\|_{L^2(\Omega)}+\beta_a^\frac{1}{2} \|\Delta u\|_{L^2(\Omega)}
\end{equation}
which concludes the proof.

\def\cprime{$'$} \def\cprime{$'$}
  \def\polhk#1{\setbox0=\hbox{#1}{\ooalign{\hidewidth
  \lower1.5ex\hbox{`}\hidewidth\crcr\unhbox0}}} \def\cprime{$'$}
  \def\polhk#1{\setbox0=\hbox{#1}{\ooalign{\hidewidth
  \lower1.5ex\hbox{`}\hidewidth\crcr\unhbox0}}}


\begin{thebibliography}{10}

\bibitem{MR2001070}
G.~Alessandrini and V.~Nesi.
\newblock Univalent {$\sigma$}-harmonic mappings: connections with
  quasiconformal mappings.
\newblock {\em J. Anal. Math.}, 90:197--215, 2003.

\bibitem{Al92}
G.~Allaire.
\newblock Homogenization and two-scale convergence.
\newblock {\em SIAM J. Math. Anal.}, 23:1482--1518, 1992.

\bibitem{MR1859696}
G.~Allaire.
\newblock {\em Shape optimization by the homogenization method}, volume 146 of
  {\em Applied Mathematical Sciences}.
\newblock Springer-Verlag, New York, 2002.

\bibitem{AllBri05}
G.~Allaire and R.~Brizzi.
\newblock A multiscale finite element method for numerical homogenization.
\newblock {\em Multiscale Model. Simul.}, 4(3):790--812 (electronic), 2005.

\bibitem{MR1892102}
A.~Ancona.
\newblock Some results and examples about the behavior of harmonic functions
  and {G}reen's functions with respect to second order elliptic operators.
\newblock {\em Nagoya Math. J.}, 165:123--158, 2002.

\bibitem{MR2231859}
T.~Arbogast and K.~J. Boyd.
\newblock Subgrid upscaling and mixed multiscale finite elements.
\newblock {\em SIAM J. Numer. Anal.}, 44(3):1150--1171 (electronic), 2006.

\bibitem{MR2342991}
T.~Arbogast, C.-S. Huang, and S.-M. Yang.
\newblock Improved accuracy for alternating-direction methods for parabolic
  equations based on regular and mixed finite elements.
\newblock {\em Math. Models Methods Appl. Sci.}, 17(8):1279--1305, 2007.

\bibitem{MR1286212}
I.~Babu{\v{s}}ka, G.~Caloz, and J.~E. Osborn.
\newblock Special finite element methods for a class of second order elliptic
  problems with rough coefficients.
\newblock {\em SIAM J. Numer. Anal.}, 31(4):945--981, 1994.

\bibitem{MR2430615}
I.~Babu{\v{s}}ka, R.~Lipton, and M.~Stuebner.
\newblock The penetration function and its application to microscale problems.
\newblock {\em BIT}, 48(2):167--187, 2008.

\bibitem{BabOsb83}
I.~Babu{\v{s}}ka and J.~E. Osborn.
\newblock Generalized finite element methods: their performance and their
  relation to mixed methods.
\newblock {\em SIAM J. Numer. Anal.}, 20(3):510--536, 1983.

\bibitem{MR1648351}
I.~Babu{\v{s}}ka and J.~E. Osborn.
\newblock Can a finite element method perform arbitrarily badly?
\newblock {\em Math. Comp.}, 69(230):443--462, 2000.

\bibitem{BaPa90}
N.~Bakhvalov and G.~Panasenko.
\newblock Homogenization: averaging processes in periodic media.
\newblock In {\em Mathematics and its applications, vol. 36}. Kluwer Academic
  Publishers, Dordrecht, 1990.

\bibitem{BeLiPa78}
A.~Bensoussan, J.~L. Lions, and G.~Papanicolaou.
\newblock {\em Asymptotic analysis for periodic structure}.
\newblock North Holland, Amsterdam, 1978.

\bibitem{MR1771781}
C.~Bernardi and R.~Verf{\"u}rth.
\newblock Adaptive finite element methods for elliptic equations with
  non-smooth coefficients.
\newblock {\em Numer. Math.}, 85(4):579--608, 2000.

\bibitem{MR2284699}
X.~Blanc, C.~Le Bris, and P.-L. Lions.
\newblock Une variante de la th\'eorie de l'homog\'en\'eisation stochastique
  des op\'erateurs elliptiques.
\newblock {\em C. R. Math. Acad. Sci. Paris}, 343(11-12):717--724, 2006.

\bibitem{MR2334772}
X.~Blanc, C.~Le Bris, and P.-L. Lions.
\newblock Stochastic homogenization and random lattices.
\newblock {\em J. Math. Pures Appl. (9)}, 88(1):34--63, 2007.

\bibitem{MR1968440}
A.~Braides.
\newblock {\em {$\Gamma$}-convergence for beginners}, volume~22 of {\em Oxford
  Lecture Series in Mathematics and its Applications}.
\newblock Oxford University Press, Oxford, 2002.

\bibitem{BraWu09}
L.~V. Branets, S.~S. Ghai, L.~L., and X.-H. Wu.
\newblock Challenges and technologies in reservoir modeling.
\newblock {\em Commun. Comput. Phys.}, 6(1):1--23, 2009.

\bibitem{BrSc02}
S.~C. Brenner and L.~R. Scott.
\newblock {\em The mathematical theory of finite elements methods}, volume~15
  of {\em Texts in Applied Mathematics}.
\newblock Springer, 2002.
\newblock Second edition.

\bibitem{MR2073507}
M.~Briane, G.~W. Milton, and V.~Nesi.
\newblock Change of sign of the corrector's determinant for homogenization in
  three-dimensional conductivity.
\newblock {\em Arch. Ration. Mech. Anal.}, 173(1):133--150, 2004.

\bibitem{MR2361302}
L.~A. Caffarelli and P.~E. Souganidis.
\newblock A rate of convergence for monotone finite difference approximations
  to fully nonlinear, uniformly elliptic {PDE}s.
\newblock {\em Comm. Pure Appl. Math.}, 61(1):1--17, 2008.

\bibitem{ChuHou09}
C.-C. Chu, I.~G. Graham, and T.~Y. Hou.
\newblock A new multiscale finite element method for high-contrast elliptic
  interface problems.
\newblock {\em Math. Comput.}, 2009.
\newblock Accepted for publication.

\bibitem{CioDon99}
Donato~P. Cioranescu~D.
\newblock An introduction to homogenization.
\newblock Oxford University Press, 1999.

\bibitem{MR1838500}
C.~Conca and M.~Vanninathan.
\newblock On uniform {$H\sp 2$}-estimates in periodic homogenization.
\newblock {\em Proc. Roy. Soc. Edinburgh Sect. A}, 131(3):499--517, 2001.

\bibitem{DeDoOw09}
M.~Desbrun, R.~Donaldson, and H.~Owhadi.
\newblock Discrete geometric structures in homogenization and inverse
  homogenization with application to eit.
\newblock {\em preprint arXiv:0904.2601}, 2009.

\bibitem{MR2314852}
W.~E, B.~Engquist, X.~Li, W.~Ren, and E.~Vanden-Eijnden.
\newblock Heterogeneous multiscale methods: a review.
\newblock {\em Commun. Comput. Phys.}, 2(3):367--450, 2007.

\bibitem{EffGaWu09}
Y.~Efendiev, J.~Galvis, and X.~Wu.
\newblock Multiscale finite element and domain decomposition methods for
  high-contrast problems using local spectral basis functions.
\newblock 2009.
\newblock Submitted.

\bibitem{MR2281625}
Y.~Efendiev, V.~Ginting, T.~Hou, and R.~Ewing.
\newblock Accurate multiscale finite element methods for two-phase flow
  simulations.
\newblock {\em J. Comput. Phys.}, 220(1):155--174, 2006.

\bibitem{MR2322432}
Y.~Efendiev and T.~Hou.
\newblock Multiscale finite element methods for porous media flows and their
  applications.
\newblock {\em Appl. Numer. Math.}, 57(5-7):577--596, 2007.

\bibitem{EnSou08}
B.~Engquist and P.~E. Souganidis.
\newblock Asymptotic and numerical homogenization.
\newblock {\em Acta Numerica}, 17:147--190, 2008.

\bibitem{Gio75}
E.~De Giorgi.
\newblock Sulla convergenza di alcune successioni di integrali del tipo
  dell'aera.
\newblock {\em Rendi Conti di Mat.}, 8:277--294, 1975.

\bibitem{MR630747}
E.~De Giorgi.
\newblock New problems in {$\Gamma $}-convergence and {$G$}-convergence.
\newblock In {\em Free boundary problems, Vol. II (Pavia, 1979)}, pages
  183--194. Ist. Naz. Alta Mat. Francesco Severi, Rome, 1980.

\bibitem{AnGlo06}
A.~Gloria.
\newblock Analytical framework for the numerical homogenization of elliptic
  monotone operators and quasiconvex energies.
\newblock {\em SIAM MMS}, 5(3):996--1043, 2006.

\bibitem{MR2399150}
H.~Harbrecht, R.~Schneider, and C.~Schwab.
\newblock Sparse second moment analysis for elliptic problems in stochastic
  domains.
\newblock {\em Numer. Math.}, 109(3):385--414, 2008.

\bibitem{MR1455261}
T.~Y. Hou and X.-H. Wu.
\newblock A multiscale finite element method for elliptic problems in composite
  materials and porous media.
\newblock {\em J. Comput. Phys.}, 134(1):169--189, 1997.

\bibitem{JiKoOl91}
V.~V. Jikov, S.~M. Kozlov, and O.~A. Oleinik.
\newblock {\em Homogenization of Differential Operators and Integral
  Functionals}.
\newblock Springer-Verlag, 1991.

\bibitem{JikKozOle94}
V.~V. Jikov, S.~M. Kozlov, and O.~A. Ole{\u\i}nik.
\newblock {\em Homogenization of differential operators and integral
  functionals}.
\newblock Springer-Verlag, Berlin, 1994.

\bibitem{OwDes09}
L.~Kharevych, P.~Mullen, H.~Owhadi, and M.~Desbrun.
\newblock Numerical coarsening of inhomogeneous elastic materials.
\newblock {\em ACM Transactions on Graphics (SIGGRAPH)}, 28(3), 2009.

\bibitem{MR961258}
N.~Kikuchi and J.~T. Oden.
\newblock {\em Contact problems in elasticity: a study of variational
  inequalities and finite element methods}, volume~8 of {\em SIAM Studies in
  Applied Mathematics}.
\newblock Society for Industrial and Applied Mathematics (SIAM), Philadelphia,
  PA, 1988.

\bibitem{Kotiuga06}
P.~R. Kotiuga.
\newblock A rationale for pursuing eit and mreit in 3-d based on weyl
  asymptotics and problem conditioning.
\newblock In {\em 12th Biennial Conference on Electromagnetic Field
  Computation, Miami, Florida, April 30-May 3, 2006}.

\bibitem{MR542557}
S.~M. Kozlov.
\newblock The averaging of random operators.
\newblock {\em Mat. Sb. (N.S.)}, 109(151)(2):188--202, 327, 1979.

\bibitem{MR1862026}
M.~Lassas and G.~Uhlmann.
\newblock On determining a {R}iemannian manifold from the
  {D}irichlet-to-{N}eumann map.
\newblock {\em Ann. Sci. \'Ecole Norm. Sup. (4)}, 34(5):771--787, 2001.

\bibitem{MR1029119}
J.~M. Lee and G.~Uhlmann.
\newblock Determining anisotropic real-analytic conductivities by boundary
  measurements.
\newblock {\em Comm. Pure Appl. Math.}, 42(8):1097--1112, 1989.

\bibitem{MR1903306}
S.~Leonardi.
\newblock Weighted {M}iranda-{T}alenti inequality and applications to equations
  with discontinuous coefficients.
\newblock {\em Comment. Math. Univ. Carolin.}, 43(1):43--59, 2002.

\bibitem{MPG00}
A.~Maugeri, D.~K. Palagachev, and L.~G. Softova.
\newblock {\em Elliptic and Parabolic Equations with Discontinuous
  Coefficients}, volume 109 of {\em Mathematical Research}.
\newblock Wiley-VCH, 2000.

\bibitem{MR1766938}
J.~M. Melenk.
\newblock On {$n$}-widths for elliptic problems.
\newblock {\em J. Math. Anal. Appl.}, 247(1):272--289, 2000.

\bibitem{MR506997}
F.~Murat.
\newblock Compacit\'e par compensation.
\newblock {\em Ann. Scuola Norm. Sup. Pisa Cl. Sci. (4)}, 5(3):489--507, 1978.

\bibitem{Mur78}
F.~Murat and L.~Tartar.
\newblock H-convergence.
\newblock {\em S\'{e}minaire d'Analyse Fonctionnelle et Num\'{e}rique de
  l'Universit\'{e} d'Alger}, 1978.

\bibitem{NetSaf05}
Y.~Netrusov and Y.~Safarov.
\newblock Weyl asymptotic formula for the laplacian on domains with rough
  boundaries.
\newblock {\em Communications in Mathematical Physics}, 253(2):481--509, 2005.

\bibitem{Ngu90}
G.~Nguetseng.
\newblock A general convergence result for a functional related to the theory
  of homogenization.
\newblock {\em SIAM J. Math. Anal.}, 20(3):608--623, 1989.

\bibitem{MR2399542}
J.~Nolen, G.~Papanicolaou, and O.~Pironneau.
\newblock A framework for adaptive multiscale methods for elliptic problems.
\newblock {\em Multiscale Model. Simul.}, 7(1):171--196, 2008.

\bibitem{OwhZha07errata}
H.~Owhadi and L.~Zhang.
\newblock Errata to metric-based upscaling.
\newblock {\em Available at http:www.acm/caltech.edu/~owhadi/}, 2007.

\bibitem{MR2292954}
H.~Owhadi and L.~Zhang.
\newblock Metric-based upscaling.
\newblock {\em Comm. Pure Appl. Math.}, 60(5):675--723, 2007.

\bibitem{MR2377253}
H.~Owhadi and L.~Zhang.
\newblock Homogenization of parabolic equations with a continuum of space and
  time scales.
\newblock {\em SIAM J. Numer. Anal.}, 46(1):1--36, 2007/08.

\bibitem{OwZh06c}
H.~Owhadi and L.~Zhang.
\newblock Homogenization of the acoustic wave equation with a continuum of
  scales.
\newblock {\em Computer Methods in Applied Mechanics and Engineering},
  198(2-4):97--406, 2008.
\newblock Arxiv math.NA/0604380.

\bibitem{PapVar82}
George~C. Papanicolaou and S.~R.~S. Varadhan.
\newblock Diffusions with random coefficients.
\newblock In {\em Statistics and probability: essays in honor of C. R. Rao},
  pages 547--552. North-Holland, Amsterdam, 1982.

\bibitem{Pinkus85}
A.~Pinkus.
\newblock {\em n-Width in Approximation Theory}.
\newblock Springer, New York, 1985.

\bibitem{MR0240443}
S.~Spagnolo.
\newblock Sulla convergenza di soluzioni di equazioni paraboliche ed
  ellittiche.
\newblock {\em Ann. Scuola Norm. Sup. Pisa (3) 22 (1968), 571-597; errata,
  ibid. (3)}, 22:673, 1968.

\bibitem{MR0477444}
S.~Spagnolo.
\newblock Convergence in energy for elliptic operators.
\newblock In {\em Numerical solution of partial differential equations, III
  (Proc. Third Sympos. (SYNSPADE), Univ. Maryland, College Park, Md., 1975)},
  pages 469--498. Academic Press, New York, 1976.

\bibitem{MR2283892}
T.~Strouboulis, L.~Zhang, and I.~Babu{\v{s}}ka.
\newblock Assessment of the cost and accuracy of the generalized {FEM}.
\newblock {\em Internat. J. Numer. Methods Engrg.}, 69(2):250--283, 2007.

\bibitem{MR2317004}
R.~A. Todor and C.~Schwab.
\newblock Convergence rates for sparse chaos approximations of elliptic
  problems with stochastic coefficients.
\newblock {\em IMA J. Numer. Anal.}, 27(2):232--261, 2007.

\bibitem{MR1511560}
H.~Weyl.
\newblock \"{U}ber gew\"ohnliche {D}ifferentialgleichungen mit
  {S}ingularit\"aten und die zugeh\"origen {E}ntwicklungen willk\"urlicher
  {F}unktionen.
\newblock {\em Math. Ann.}, 68(2):220--269, 1910.

\bibitem{WhHo87}
C.~D. White and R.~N. Horne.
\newblock Computing absolute transmissibility in the presence of finescale
  heterogeneity.
\newblock {\em SPE Symposium on Reservoir Simulation}, page 16011, 1987.

\bibitem{MR1898136}
X.~H. Wu, Y.~Efendiev, and T.~Y. Hou.
\newblock Analysis of upscaling absolute permeability.
\newblock {\em Discrete Contin. Dyn. Syst. Ser. B}, 2(2):185--204, 2002.

\bibitem{ZhBFOww09}
L.~Zhang, L.~Berlyand, M.~Federov, and H.~Owhadi.
\newblock Global energy matching method for atomistic to continuum modeling of
  self-assembling biopolymer aggregates.
\newblock {\em submitted}, 2009.

\end{thebibliography}
\end{document}